\documentclass{amsart}

\usepackage{amssymb,amsmath, amscd, mathrsfs, graphics}

\newtheorem{thm}{Theorem}[section]

\newtheorem{lemma}[thm]{Lemma}
\newtheorem{remark}[thm]{Remark}
\newtheorem{prop}[thm]{Proposition}

\title[Noncommutative Julia-Wolff-Carath\'eodory Theorem]{A noncommutative
version of the Julia-Wolff-Carath\'eodory Theorem}
\author{Serban Teodor Belinschi}
\address{CNRS, Institut de Math\'ematiques de Toulouse\\
118 Route de Narbonne\\
F-31062 Toulouse Cedex 09, France.}
\email{serban.belinschi@math.univ-toulouse.fr}
\thanks{}

\begin{document}

\begin{abstract}
The classical Julia-Wolff-Carath\'eodory Theorem characterizes the behaviour
of the derivative of an analytic self-map of a unit disc or of a 
half-plane of the complex plane at certain boundary points. We prove a 
version of this result that applies to noncommutative self-maps of 
noncommutative half-planes in von Neumann algebras at points of the
distinguished boundary of the domain. Our result, somehow 
surprisingly, relies almost entirely on simple geometric properties of 
noncommutative half-planes, which are quite similar to the geometric 
properties of classical hyperbolic spaces.

\end{abstract}

\maketitle

\section{Introduction}

The classical Julia-Wolff-Carath\'eodory Theorem describes the behaviour of 
the derivative of an analytic self-map of the unit disc $\mathbb D$
or of the upper half-plane $\mathbb C^+$ of the complex plane $\mathbb 
C$ at certain boundary points. 
Numerous generalizations, to self-maps of balls or polydisks in 
$\mathbb C^n$, analytic functions with values in spaces of linear
operators, analytic self-maps on domains in spaces of operators or
in more general Banach spaces etc. - see for example
\cite{Rudin,Fan,Jafari,Wlo,Abate,Mellon,MM,AbateRaissy} (the list is not exhaustive) -
are known. This note gives a version of this theorem for noncommutative
self-maps of the noncommutative upper half-plane of a von Neumann 
algebra $\mathcal A$. The result builds on the recent literature in the
field - see \cite{AMcY,ATY2,PTD}, and falls under the programme aiming
to find the noncommutative versions of classical complex analysis results - see
for example \cite{AKV0,AKV,AMc1,AMc3,AMc2}.

In the second section we state our main result, and provide the 
required background. The third section is dedicated to proving
a Schwarz lemma-type result for noncommutative functions. In
this same section we give a simple (not necessarily original,
though) proof of the classical Julia-Wolff-Carath\'eodory Theorem 
in order to make this article self-contained, and some lemmas that
make use of it. Finally, in the last section we prove our main result.

\section{Noncommutative functions and the Julia-Carath\'eodory Theorem}

\subsection{Noncommutative functions}
Noncommutative sets and functions originate in \cite{taylor0,taylor}. We 
largely follow \cite{ncfound} in our presentation below. We refer to this
excellent monograph for details on, and proofs of, the statements 
that follow.

First a notation: if $S$ is a nonempty set, we denote by $M_{m\times n}
(S)$ the set of all matrices with $m$ rows and $n$ columns having 
entries from $S$. For simplicity, we let $M_n(S):=M_{n\times n}(S)$. 
Given C${}^*$-algebra $\mathcal A$, a {\em noncommutative set} is a 
family $\Omega:=(\Omega_n)_{n\in\mathbb N}$
such that
\begin{enumerate}
\item[(a)] for each $n\in\mathbb N$, $\Omega_n\subseteq M_n
(\mathcal A);$
\item[(b)] for each $m,n\in\mathbb N$, we have $\Omega_m\oplus
\Omega_n\subseteq\Omega_{m+n}$.
\end{enumerate}
The noncommutative set $\Omega$ is called {\em right admissible} if in
addition the condition (c) below is satisfied:
\begin{enumerate}
\item[(c)] for each $m,n\in\mathbb N$ and $a\in\Omega_m,c\in
\Omega_n,w\in M_{m\times n}(\mathcal A)$, there is an $\epsilon>0$
such that $\left(\begin{array}{cc}
a & zw\\
0 & c\end{array}\right)\in\Omega_{m+n}$ for all $z\in\mathbb C,
|z|<\epsilon$.
\end{enumerate}
Given C${}^*$-algebras $\mathcal{A,C}$ and a noncommutative
set $\Omega\subseteq\coprod_{n=1}^\infty M_n(\mathcal A)$, a {\em 
noncommutative function} is a family $f:=(f_n)_{n\in\mathbb N}$
such that $f_n\colon\Omega_n\to M_n(\mathcal C)$ and
\begin{enumerate}
\item $f_m(a)\oplus f_n(c)=f_{m+n}(a\oplus c)$ for all
$m,n\in\mathbb N$, $a\in\Omega_m,c\in\Omega_n$;
\item for all $n\in\mathbb N$, $f_n(T^{-1}aT)=T^{-1}f_n(a)T$ whenever
$a\in\Omega_n$ and $T\in GL_n(\mathbb C)$ are such that $T^{-1}aT$
belongs to the domain of definition of $f_n$.
\end{enumerate}
These two conditions are equivalent to the single condition
\begin{enumerate}
\item[(A)] For any $m,n\in\mathbb N$, $a\in\Omega_m,c\in\Omega_n$, 
$S\in M_{m\times n}(\mathbb C)$, one has
$$
aS=Sc\implies f_m(a)S=Sf_n(c).
$$
\end{enumerate}
We shall refer to the indices $n$ of $\Omega_n$ or of $f_n$ as the ``levels''
of the noncommutative set $\Omega$ or of the noncommutative function $f$.

A remarkable result (see \cite[Theorem 7.2]{ncfound}) states that, 
under very mild conditions on $\Omega$, local boundedness for $f$
implies each $f_n$ is analytic as a map between Banach spaces. Indeed,
a hint towards the proof of this result is the following essential
property of noncommutative functions: if $\Omega$ is admissible,
$a\in \Omega_n, c\in\Omega_m, b\in M_{n\times m}
(\mathcal A)$, such that $\left(\begin{array}{cc}
a & b \\
0 & c
\end{array}\right)\in\Omega_{n+m}$, then there exists a linear map
$\Delta f_{n,m}(a,c)\colon M_{n\times m}(\mathcal A)\to M_{n\times m}
(\mathcal C)$ such that 
\begin{equation}\label{FDQ}
f_{n+m}\left(\begin{array}{cc}
a & b \\
0 & c
\end{array}\right)=\left(\begin{array}{cc}
f_n(a) & \Delta f_{n,m}(a,c)(b) \\
0 & f_m(c)
\end{array}\right).
\end{equation}
Obviously, this implies in particular that $f_{n+m}$ extends to
the set of all elements $\left(\begin{array}{cc}
a & b \\
0 & c
\end{array}\right)$ such that $a\in \Omega_n, c\in\Omega_m, 
b\in M_{n\times m}(\mathcal A)$ (see \cite[Section 2.2]{ncfound}).
Two properties of this operator that are important for us are
\begin{equation}\label{FDC}
\Delta f_{n,n}(a,c)(a-c)=f(a)-f(c)=\Delta f_{n,n}(c,a)(a-c),\quad
 \Delta f_{n,n}(a,a)(b)=f_n'(a)(b),
\end{equation}
the classical Frechet derivative of $f_n$ in $a$ aplied to the element
$b\in M_n(\mathcal A)$. Moreover, $\Delta f_{n,m}(a,c)$ as functions of $a$
and $c$, respectively, satisfy properties similar to the ones
described in items (1), (2) above (see \cite[Sections 2.3--2.5]{ncfound}
for details).
For convenience, from now on we shall suppress the indices denoting
the levels for noncommutative functions, as it will almost always
be obvious from the context.

We provide three examples of noncommutative sets: 
\begin{enumerate}
\item[(i)] The noncommutative upper half-plane 
$H^+(\mathcal A)=(H^+(M_n(\mathcal A)))_{n\in\mathbb N}$, 
where $H^+(M_n(\mathcal A))=\{b\in M_n(\mathcal A)\colon\Im b>0\}$ (here
$\Im b=\frac{b-b^*}{2i},\Re b=\frac{b+b^*}{2}$), 
\item[(ii)] The set of nilpotent matrices with entries from $\mathcal A
$, and
\item[(iii)] The unit ball $(B(M_n(\mathcal A)))_{n\in\mathbb N}$, 
where $B(M_n(\mathcal A))=\{b\in M_n(\mathcal A)\colon\|b\|<1\}$.
\end{enumerate}
Our paper will focus on the first example. 

As the domains we consider in this paper are mostly 
noncommutative subsets of von Neumann algebras given
by an order relation, we give below a few of the well-known
results we use systematically in the rest of the paper. For them, we refer 
to \cite{Bruce,Paulsen,SZ}. First, recall that for any C${}^*$-algebra
(hence, in particular, von Neumann algebra) $\mathcal A$, if
$x\in\mathcal A$, then $\|x\|^2=\|x^*\|^2=\|x^*x\|=\|xx^*\|$.
For a selfadjoint element $x=x^*\in\mathcal A$, $\|x\|$ is equal
to the spectral radius of $x$. We say that $x\ge0$ in $\mathcal A$
if $x=x^*$ and the spectrum of $x$ is included in $[0,+\infty)$.
Equivalently, if $\mathcal H$ is the Hilbert space on which 
$\mathcal A$ acts as a von Neumann algebra, then a selfadjoint
$x\in\mathcal A$ is greater than or equal to zero if  and
only if $\langle x\xi,\xi\rangle\ge0$ for all $\xi\in\mathcal H$.
We say that $x>0$ means that $x\ge0$ and $x$ is invertible (i.e.
its spectrum is included in $(0,+\infty)$). We say $x\ge y$
if $x-y\ge0$, and similarly for ``$>$.'' In particular,
it follows that $xx^*\leq\|x\|^2\cdot1_\mathcal A$
and $x^*x\leq\|x\|^2\cdot1_\mathcal A$. Clearly, for 
and $\varepsilon\in(0,+\infty)$, 
$xx^*<(\|x\|^2+\varepsilon)\cdot1_\mathcal A$, with strict inequality
achieved only when $\varepsilon>0$.

As proved in \cite[Lemma 3.1]{Paulsen}, 
$$
\left(\begin{array}{cc}
1 & a\\
a^* & 1
\end{array}\right)\ge0\text{ in }M_2(\mathcal A)\iff \|a\|\leq1.
$$
We claim that 
$$
\left(\begin{array}{cc}
1 & a\\
a^* & 1
\end{array}\right)>0 \text{ in }M_2(\mathcal A)\iff \|a\|<1.
$$
Indeed, if $\|a\|<1$, then $1-aa^*$ and $1-a^*a$ are 
invertible in $\mathcal A$ and 
$$
\left(\begin{array}{cc}
1 & a\\
a^* & 1
\end{array}\right)^{-1}=\left(\begin{array}{cc}
(1-aa^*)^{-1} & -(1-aa^*)^{-1}a\\
-a^*(1-aa^*)^{-1} & (1-a^*a)^{-1}
\end{array}\right)\ \text{in }M_2(\mathcal A).
$$
Conversely, if $\|a\|=1$, then for any $\varepsilon>0$ there exists
$\xi_\varepsilon,\eta_\varepsilon\in\mathcal H$ of norm one such that 
$\langle a\eta_\varepsilon,\xi_\varepsilon\rangle>
1-\varepsilon.$ Then 
$$
\left|\left\langle
\left(\begin{array}{cc}
1 & a\\
a^* & 1
\end{array}\right)\left[\begin{array}{c}
\xi_\varepsilon\\
\eta_\varepsilon
\end{array}\right],\left[\begin{array}{c}
\xi_\varepsilon\\
\eta_\varepsilon
\end{array}\right]
\right\rangle\right|=\|\xi_\varepsilon\|^2_2+\|\eta_\varepsilon\|^2_2
-2\langle a\eta_\varepsilon,\xi_\varepsilon\rangle<2\varepsilon,
$$
so that zero belongs to the spectrum of $\left(\begin{array}{cc}
1 & a\\
a^* & 1
\end{array}\right)$. This proves our claim.

Observe also that for any selfadjoint $x\in\mathcal A$, we have $x>0$ 
if and only if for any invertible $a\in\mathcal A$,
we have $a^*xa>0$. Indeed, one implication is obvious. Conversely,
if $a$ is invertible and $a^*xa>0$, then there is an $\varepsilon_a\in(0,+\infty)$ 
such that $a^*xa>\varepsilon_a\cdot1_\mathcal A$.
For any $\xi\in\mathcal H$, 
$\langle x\xi,\xi\rangle=\langle x
a(a^{-1}\xi),a(a^{-1}\xi)\rangle=\langle a^*x
a(a^{-1}\xi),(a^{-1}\xi)\rangle>\varepsilon_a\|(a^{-1}\xi)\|_2^2\ge
\varepsilon_a\|a\|^{-2}\|\xi\|_2^2,$ independently of $\xi$, so that 
$x\ge\varepsilon_a\|a\|^{-2}\cdot1_\mathcal A>0.$ 
We use these last two results to conclude that
$$
\left(\begin{array}{cc}
u & v\\
v^* & w
\end{array}\right)>0\text{ in }M_2(\mathcal A)\iff u,w>0\text{ in }\mathcal A\text{ and }
\left\{\begin{array}{c}
u>vw^{-1}v^* \\
\text{and/or}\\
w>v^*u^{-1}v 
\end{array}\right..
$$
This follows from the above by writing 
$$
\left(\begin{array}{cc}
u & v\\
v^* & w
\end{array}\right)=
\left(\begin{array}{cc}
u^\frac12 & 0\\
0 & w^\frac12
\end{array}\right)
\left(\begin{array}{cc}
1 & u^{-\frac12}vw^{-\frac12}\\
w^{-\frac12}v^*u^{-\frac12} & 1
\end{array}\right)
\left(\begin{array}{cc}
u^\frac12 & 0\\
0 & w^\frac12
\end{array}\right)
$$
and recalling the chain of equivalences $\|u^{-1/2}vw^{-1/2}(u^{-1/2}vw^{-1/2})^*\|<1\iff
\|(u^{-1/2}vw^{-1/2})^*u^{-1/2}vw^{-1/2}\|<1\iff
u^{-1/2}vw^{-1/2}(u^{-1/2}vw^{-1/2})^*<1\iff (u^{-1/2}vw^{-1/2})^*u^{-1/2}vw^{-1/2}<1.$
We shall use these facts below without further referencing to them.

\subsection{The Julia-Wolff-Carath\'eodory Theorem, classical and noncommutative}
We state the classical Julia-Wolff-Carath\'eodory Theorem for analytic
self-maps of the upper half-plane $\mathbb C^+$ at a point of
the real line $\mathbb R$. In the following we denote by $\displaystyle
\lim_{\stackrel{ z\longrightarrow\alpha}{{\sphericalangle}}}$ the nontangential 
limit at a point $\alpha\in\mathbb R$ (see, for ex. \cite{garnett}).
\begin{thm}\label{JC}
Let $f\colon\mathbb C^+\to\mathbb C^+$ be a nonconstant analytic function and $\alpha\in
\mathbb R$ be fixed. 
\begin{enumerate}
\item Assume that
\begin{equation}\label{3}
c:=\liminf_{z\to\alpha}\frac{\Im f(z)}{\Im z}<\infty.
\end{equation}
Then $f(\alpha):=\displaystyle\lim_{\stackrel{ z\longrightarrow\alpha}{{
\sphericalangle}}}f(z)$ exists and belongs to
$\mathbb R$, and 
\begin{equation}\label{4}
\lim_{\stackrel{ z\longrightarrow\alpha}{{
\sphericalangle}}}f'(z)=
\lim_{\stackrel{ z\longrightarrow\alpha}{{
\sphericalangle}}}\frac{f(z)-f(\alpha)}{z-\alpha}=
\lim_{y\downarrow0}\frac{\Im f(\alpha+iy)}{y}=c.
\end{equation}
\item Assume that $\displaystyle\lim_{y\downarrow0}f(\alpha+iy)=:
f(\alpha)$ exists and belongs to $\mathbb R$. If 
$$
\lim_{\stackrel{ z\longrightarrow\alpha}{{
\sphericalangle}}}\frac{f(z)-f(\alpha)}{z-\alpha}=c\in\mathbb C,
$$
then $c\in(0,+\infty)$ and 
$$
c=\liminf_{z\to\alpha}\frac{\Im f(z)}{\Im z}=
\lim_{\stackrel{ z\longrightarrow\alpha}{{
\sphericalangle}}}f'(z).
$$
\item Assume that $\displaystyle\lim_{\stackrel{z\longrightarrow
\alpha}{{\sphericalangle}}}f'(z)=c\in\mathbb C$ and 
$\displaystyle\lim_{\stackrel{ z\longrightarrow\alpha}{{
\sphericalangle}}}f(z)=f(\alpha)\in\mathbb R$. Then
$$
c=\liminf_{z\to\alpha}\frac{\Im f(z)}{\Im z}=
\lim_{\stackrel{ z\longrightarrow\alpha}{{
\sphericalangle}}}\frac{f(z)-f(\alpha)}{z-\alpha}\in\mathbb R.
$$
\end{enumerate}
\end{thm}
The noncommutative version of this theorem becomes quite
obvious in light of \eqref{FDC} and of the formulations of the
corresponding main result from \cite{Wlo}  as well as the recent work
\cite{PTD}. In the following, when
we make a statement about a completely positive map, we usually write
the statement for level one, and, unless the contrary is explicitly
stated, we mean that the property in question holds for all levels $n$.
Thus, for example, the statement
$$
\lim_{\stackrel{ z\longrightarrow0}{{
\sphericalangle}}}f'(\alpha+zv):=f'(\alpha)
$$
exists and is completely positive for $\alpha=\alpha^*\in\mathcal A$ and $v>0$ 
means that for any $n\in\mathbb N$ and any $v\in M_n(\mathcal A)$, 
$$
\lim_{\stackrel{ z\longrightarrow0}{{
\sphericalangle}}}f'(\alpha\otimes1_n+zv)=f'(\alpha\otimes1_n)=
f'(\alpha)\otimes{\rm Id}_n
$$
is a positive map on $M_n(\mathcal A)$.

\begin{thm}\label{Main}
Let $\mathcal A$ be a von Neumann algebra and $f\colon
H^+(\mathcal A)\to H^+(\mathcal A)$ be a noncommutative analytic map. Fix 
$\alpha=\alpha^*\in\mathcal A$.
\begin{enumerate}
\item Assume that for any $v\in \mathcal A,
v>0$ and any state $\varphi$ on $\mathcal A$,
\begin{equation}\label{5}
\liminf_{z\to0,z\in\mathbb C^+}\frac{\varphi(\Im f(\alpha
+zv))}{\Im z}<\infty.
\end{equation}
Then there exists $c=c(v)\in\mathcal A$, $c>0$ such that
\begin{equation}\label{6}
\lim_{y\downarrow0}\frac{\Im f(\alpha+iyv)}{y}=c
\end{equation}
in the strong operator {\rm (so)} topology. Moreover, $\displaystyle\lim_{\stackrel{z\longrightarrow0}{{\sphericalangle}}}
f(\alpha+zv)=f(\alpha)$ exists, does not depend on $v$ and is selfadjoint. The limits
\begin{equation}\label{7}
\lim_{\stackrel{ z\longrightarrow0}{{
\sphericalangle}}}\Delta f(\alpha+zv_1,\alpha+zv_2)\quad\text{and}\quad
\lim_{\stackrel{z\longrightarrow0}{{\sphericalangle}}}f'(\alpha+zv)
\end{equation}
exist pointwise  {\rm wo} for any $v,v_1,v_2>0$, and
$\displaystyle\lim_{\stackrel{ z\longrightarrow0}{{
\sphericalangle}}}f'(\alpha+zv)(v)=c(v)$. All statements remain true for any $n\in\mathbb N$,
$v,v_1,v_2>0$ in $M_n(\mathcal A)$ and $\alpha$ replaced by $\alpha\otimes 1_n$.

\item[(1')] Assume in addition to the hypothesis \eqref{5} that for any $v,w>0$
in $\mathcal A$ and any state $\varphi$ on $\mathcal A$, the gradient of the two-variable 
complex function $\{(z,\zeta)\in\mathbb C^2\colon\Im(zv+\zeta w)>0\}\ni(z,\zeta)\mapsto\varphi(f(\alpha+zv+\zeta w))
\in\mathbb C^+$ admits the limit
$$
\lim_{\stackrel{y_1,y_2\downarrow0}{(y_1,y_2)\in[0,1)^2\setminus\{(0,0)\}}}
\left(\varphi(f'(\alpha+iy_1v+iy_2w)(v)),\varphi(f'(\alpha+iy_1v+iy_2w)(w))\right).
$$
Then the limits \eqref{7}
are equal to each other, completely positive and do not depend on $v,v_1,v_2$.

\item Assume that the pointwise {\rm wo} limit
$\displaystyle\lim_{y\downarrow0}f'(\alpha+iyv):=f'(\alpha)$
exists for any $v>0$, does not depend on $v$ and $f'(\alpha)$ is
a completely bounded operator on $\mathcal A$. Then $f'(\alpha)$
is completely positive, $\displaystyle\lim_{\stackrel{z
\longrightarrow0}{{\sphericalangle}}}f(\alpha+zv):=f(\alpha)$
exists, does not depend on $v$ and is selfadjoint, and 
$$
f'(\alpha)(v)={\rm so\text{-}}
\lim_{\stackrel{ z\longrightarrow0}{{
\sphericalangle}}}\frac{\Im f(\alpha+iyv)}{y}\quad\text{for any }
v>0.
$$

\end{enumerate}
\end{thm}

Unfortunately, unlike in the classical case of Theorem \ref{JC}, and similar to
the case of functions of several complex variables \cite{Rudin,Abate}, item (1') 
above cannot be improved upon. Indeed, it was observed in \cite{AMcY} that
for analytic functions of two complex variables on the bidisk with values
in the unit disk, there exist examples that satisfy the commutative
equivalent of \eqref{5} for the bidisk, and yet 
the gradient map does not have a nontangential limit. 
The equivalent of condition \eqref{5} implies the existence
of all {\em directional} derivatives in permissible directions, but 
these directional derivatives do not necessarily ``add up'' to a linear map.
This commutative example has a natural noncommutative extension,
as shown in \cite{PTD}. It is enough for our purposes to treat
a simplified version of this extension. It is shown in \cite{ATY} that any
Loewner map from the $n$-dimensional upper half-plane $(\mathbb C^+)^n$
to $\mathbb C^+$ has a certain operatorial realization: for any such
$h\colon(\mathbb C^+)^n\to\mathbb C^+$ there exist Hilbert spaces
$\mathcal N,\mathcal M$, a selfadjoint densely defined operator $A$ on
$\mathcal M$, a real number $s$ an orthogonal decomposition $P=
\{P_1,\dots,P_n\}$ of $\mathcal N\oplus\mathcal M$ 
(i.e. $P_iP_j=P_jP_i=\delta_{ij}P_j=\delta_{ij}P_j^*$
and $P_1+\cdots+P_n=1_{\mathcal M\oplus\mathcal N}$) and a vector state
$\varphi_v(\cdot)=\langle\cdot v,v\rangle$ on the von Neumann algebra of 
bounded linear operators on $\mathcal N\oplus\mathcal M$ such that
$$
h(z)=s+\varphi_v(M(z)), \quad z=(z_1,\dots,z_n)\in(\mathbb C^+)^n,1\leq j\leq n,
$$
where 
\begin{eqnarray*}
M(z) & = & \left(\begin{array}{cc}
-i & 0 \\
0 & 1-iA
\end{array}\right)\left(\left(\begin{array}{cc}
1 & 0 \\
0 & A
\end{array}\right)-(z_1P_1+\cdots+z_nP_n)
\left(\begin{array}{cc}
0 & 0 \\
0 & 1
\end{array}\right)\right)^{-1}\\
& & \mbox{}\times\left((z_1P_1+\cdots+z_nP_n)
\left(\begin{array}{cc}
1 & 0 \\
0 & A
\end{array}\right)+
\left(\begin{array}{cc}
0 & 0 \\
0 & 1
\end{array}\right)\right)
\left(\begin{array}{cc}
-i & 0 \\
0 & 1-iA
\end{array}\right)^{-1}.
\end{eqnarray*}
The $2\times2$ matrix decomposition is realized with respect to the canonical 
orthogonal decomposition of $\mathcal N\oplus\mathcal M$. We observe
that such maps $M\colon(\mathbb C^+)^n\to B(\mathcal N\oplus\mathcal M)$
have a natural noncommutative extension to $H^+(\mathbb C^n):=
\coprod_{k\ge1}\{a\in M_k(\mathbb C)\colon \Im a>0\}^n$ given by replacing 
$(z_1P_1+\cdots+z_nP_n)$ in the above formula of $M(z)$ by
$$
\sum_{j=1}^n
(P_j\otimes1_k) a_j 
(P_j\otimes 1_k).
$$
(While it is not obvious from its formula that $\Im M$ is positive when evaluated on 
$(\mathbb C^+)^n$, and even less when its amplification is evaluated on
$\{a\in M_k(\mathbb C)\colon \Im a>0\}^n$, a careful reading of the
proofs of \cite[Propositions 3.4 and 3.5]{ATY} allows one to observe that
they adapt without modification to show that $\Im M(a_1,\dots,
a_n)>0$ for $(a_1,\dots,a_n)\in\{a\in M_k(\mathbb C)\colon \Im a>0\}^n$.)
The extension of $h$ becomes
$$
h_k(a)=s\otimes1_k+(\varphi_v\otimes{\rm Id}_k)(M(a)),$$
for all $ a=(a_1,\dots,a_n)\in
\{a\in M_k(\mathbb C)\colon \Im a>0\}^n.$ For $n=2$ {\em any} analytic 
function $h\colon\mathbb C^+\times\mathbb C^+\to\mathbb C^+$ 
admits such an operatorial realization, and hence it has a noncommutative
extension as described above (see \cite{AMcY,ATY,ATY2}).
Considering the counterexample $h$ provided in \cite{AMcY},
the map $H\colon H^+(\mathbb C^2)\to H^+(\mathbb C^2)$ defined
by $H(a)=(h(a),h(a))$ shows that we cannot dispense of item (1') in
Theorem \ref{Main}. However, observe that the noncommutative 
structure of the function $f$ in Theorem \ref{Main} (1) allows for
a slightly stronger conclusion than in classical case of \cite{AMcY,Abate}: the
``directional derivative'' becomes a bounded linear operator
defined on all of $\mathcal A$.

As noted above, a classical analytic function is itself the first level
of a noncommutative function, via the classical analytic functional
calculus applied to matrices over $\mathbb C$. Relations \eqref{5}, 
\eqref{6}, \eqref{7} are the obvious consequences of relations \eqref{3}
and \eqref{4} in this context. Thus the statements of Theorem 
\ref{Main} are anything but surprising. Indeed, if
$f$ has an analytic extension around $\alpha$, then
the proof of Theorem \ref{Main} is absolutely trivial.

%%%%%%%%%%%%%%%%%%%%%%%%%%%%%%%%%%%%%%%%%%%%%%%%%%%%%%%%%%%%
%%%%%%%%%%%%%%%%%%%%%%%%%%%%%%%%%%%%%%%%%%%%%%%%%%%%%%%%%%%%

\section{A norm estimate}

Several slightly different proofs of Julia-Wolff-Carath\'eodory Theorem 
can be found in the literature. An essential element in one of them is 
the Schwarz-Pick Lemma: an analytic self-map of the upper half-plane
is a contraction with respect to a ``good'' metric on $\mathbb C^+$.
In the next proposition, we obtain a similar result for noncommutative
functions. We think that there is a rather striking resemblance between 
our result below and \cite[Corollary 3.3]{Mellon}, but it is not clear to us 
yet whether the two results can be obtained from each other, or even to
what extent they are related. We intend to pursue this question later.

\begin{prop}\label{prop:3.1}
Let $\mathcal A,\mathcal C$ be two von Neumann algebras and $f\colon H^+(\mathcal A)
\to H^+(\mathcal C)$ be a noncommutative map. For any $n\in\mathbb N$ and $a,c\in 
H^+(M_n(\mathcal A))$, the linear map
$$
M_n(\mathcal A)\ni b\mapsto\left(\Im f(a)\right)^{-\frac12}\Delta f
(a,c)\left((\Im a)^{\frac12}b(\Im c)^\frac12\right)
\left(\Im f(c)\right)^{-\frac12}\in M_n(\mathcal C)
$$
is a complete contraction. In particular,
$$
\left\|\left(\Im f(a)\right)^{-\frac12}\Delta f
(a,c)(b)\left(\Im f(c)\right)^{-\frac12}\right\|_\mathcal C\leq
\left\|\left(\Im a\right)^{-\frac12}b\left(\Im c\right)^{-\frac12}\right\|_\mathcal A,
$$
so that, by Equation \eqref{FDC}, for $b=a-c$,
$$
\left\|\left(\Im f(a)\right)^{-\frac12}(f(a)-f(c))
\left(\Im f(c)\right)^{-\frac12}\right\|_\mathcal C\leq\left\|\left(
\Im a\right)^{-\frac12}(a-c)\left(\Im c\right)^{-\frac12}\right\|_\mathcal A.
$$
\end{prop}
The estimate will often be used under the equivalent forms 
\begin{eqnarray}
\lefteqn{\left[(\Im f(a))^{-\frac12}\Delta f(a,c)(b)(\Im f(c))^{-\frac12}\right]^*
\left[(\Im f(a))^{-\frac12}\Delta f(a,c)(b)(\Im f(c))^{-\frac12}\right]}\nonumber\\ & \leq & 
\left\|\left(\Im a\right)^{-\frac12}b\left(\Im c\right)^{-\frac12}\right\|_\mathcal A^2\cdot
1_{M_n(\mathcal C)},\label{est}
\quad\quad\quad\quad\quad\quad\quad\quad\quad\quad\quad\quad\quad\quad\quad\quad\quad
\end{eqnarray}
\begin{eqnarray}
\lefteqn{\left[(\Im f(a))^{-\frac12}\Delta f(a,c)(b)(\Im f(c))^{-\frac12}\right]
\left[(\Im f(a))^{-\frac12}\Delta f(a,c)(b)(\Im f(c))^{-\frac12}\right]^*}\nonumber\\ & \leq & 
\left\|\left(\Im a\right)^{-\frac12}b\left(\Im c\right)^{-\frac12}\right\|_\mathcal A^2\cdot
1_{M_n(\mathcal C)},\label{est'}
\quad\quad\quad\quad\quad\quad\quad\quad\quad\quad\quad\quad\quad\quad\quad\quad\quad
\end{eqnarray}
\begin{equation}
\Delta f(a,c)(b)^*(\Im f(a))^{-1}
\Delta f(a,c)(b)\leq
\left\|\left(\Im a\right)^{-\frac12}b\left(\Im c\right)^{-\frac12}\right\|_\mathcal A^2\cdot
\Im f(c),\label{est''}
\end{equation}
\begin{equation}
\Delta f(a,c)(b)(\Im f(c))^{-1}
\Delta f(a,c)(b)^*\leq
\left\|\left(\Im a\right)^{-\frac12}b\left(\Im c\right)^{-\frac12}\right\|_\mathcal A^2\cdot
\Im f(a),\label{est'''}
\end{equation}
which we give here for convenience. Of course, if $b\in M_n(\mathcal A)$, the notation 
$\|b\|_\mathcal A$ signifies the C${}^*$-norm of $b$ as an element in $M_n(\mathcal A)$.
\begin{proof}
As 
$\Im \left(\begin{array}{cc}
a & b\\
0 & c
\end{array}\right)=
\left(\begin{array}{cc}
\Im a & \frac{b}{2i}\\
\left(\frac{b}{2i}\right)^* & \Im c
\end{array}\right)$,
%iff $u>0,w>0$ and $v^*u^{-1}v<w$, or, equivalently, if
%$u>0,w>0$ and $u>vw^{-1}v^*$ (see \cite{Paulsen}). 
we have $\left(\begin{array}{cc}
a & b\\
0 & c
\end{array}\right)\in H^+(M_{2n}(\mathcal A))$ if and only if
$a,c\in H^+(M_{n}(\mathcal A))$ and $b^*(\Im a)^{-1}b<4\Im c$.
This last relation is equivalent to $\left[(\Im a)^{-\frac12}b
(\Im c)^{-\frac12}\right]^*\left[(\Im a)^{-\frac12}b
(\Im c)^{-\frac12}\right]<4$, or $\|(\Im a)^{-\frac12}b
(\Im c)^{-\frac12}\|_\mathcal A<2$. Thus, as $f$ maps the noncommutative
upper half-plane into itself, and for any $b_0\in M_n(\mathcal A)$
there exists an $\varepsilon_{b_0}=\frac{2}{\|(\Im a)^{-\frac12}b_0
(\Im c)^{-\frac12}\|_\mathcal A}>0$ such that
$$
\left(\begin{array}{cc}
a & \varepsilon {b_0}\\
0 & c
\end{array}\right)\in H^+(M_{2n}(\mathcal A))\quad\text{for all }
\varepsilon\in[0,\varepsilon_{b_0}),
$$
and so
$$
\left(\begin{array}{cc}
f(a) & \varepsilon\Delta f(a,c)({b_0})\\
0 & f(c)
\end{array}\right)\in H^+(M_{2n}(\mathcal A))\quad\text{for all }
\varepsilon\in[0,\varepsilon_{b_0}).
$$
In particular
$\varepsilon\left\|\left(\Im f(a)\right)^{-\frac12}
\Delta f(a,c)({b_0})\left(\Im f(c)\right)^{-\frac12}\right\|_\mathcal C<2$
for $\varepsilon<\frac{2}{\|(\Im a)^{-\frac12}b_0
(\Im c)^{-\frac12}\|_\mathcal A}$. Letting 
$\varepsilon\to\frac{2}{\|(\Im a)^{-\frac12}b_0
(\Im c)^{-\frac12}\|_\mathcal A}$ from below, we obtain
$$
\left\|\left(\Im f(a)\right)^{-\frac12}
\Delta f(a,c)({b_0})\left(\Im f(c)\right)^{-\frac12}\right\|_\mathcal C\leq
\|(\Im a)^{-\frac12}b_0(\Im c)^{-\frac12}\|_\mathcal A.
$$
As $b_0\in M_n(\mathcal A)$ has been chosen arbitrarily, we
can replace it by $(\Im a)^{\frac12}b(\Im c)^{\frac12}$
to conclude that, as claimed
$$
\left\|\left(\Im f(a)\right)^{-\frac12}\Delta f
(a,c)\left((\Im a)^{\frac12}b(\Im c)^\frac12\right)
\left(\Im f(c)\right)^{-\frac12}\right\|_\mathcal C\leq\|b\|_\mathcal A,\quad b\in
M_n(\mathcal A).
$$
\end{proof}

Clearly, the same method can be used to obtain estimates involving
$\Delta^jf$ for all $j\in\mathbb N$. 
We give one such estimate pertaining to a special case of $j=2$. 
We shall simply apply the above result to appropriately
chosen elements in $M_2(\mathcal A)$. Let
$$
\left(\begin{array}{cccc}
a_1 & 0 & 0 & 0 \\
0 & a_2 & c & 0 \\
0 & 0 & a_3 & b \\
0 & 0 & 0 & a_4 
\end{array}\right)
$$
be such that $\Im a_j>0$ and $
\left(\begin{array}{cc}
a_3 & b\\
0 & a_4
\end{array}\right)\in H^+(M_2(\mathcal A))$. 
From \cite[Section 3]{ncfound} we obtain
\begin{eqnarray*}
\lefteqn{f\left(\begin{array}{cccc}
a_1 & 0 & 0 & 0 \\
0 & a_2 & c & 0 \\
0 & 0 & a_3 & b \\
0 & 0 & 0 & a_4 
\end{array}\right)}\\
& = &
\left(\begin{array}{cccc}
f(a_1) & 0 & 0 & 0 \\
0 & f(a_2) & \Delta f(a_2,a_3)(c) & \Delta^2 f(a_2,a_3,a_4)(c,b) \\
0 & 0 & f(a_3) & \Delta f(a_3,a_4)(b) \\
0 & 0 & 0 & f(a_4) 
\end{array}\right).\quad\quad
\end{eqnarray*}
Applying Proposition \ref{prop:3.1} to $a=\left(\begin{array}{cc}
a_1 & 0\\
0 & a_2
\end{array}\right),c=\left(\begin{array}{cc}
a_3 & b\\
0 & a_4
\end{array}\right)$ and $b=\left(\begin{array}{cc}
0 & 0\\
c & 0
\end{array}\right)$
under the form of \eqref{est'''} provides an estimate for $\Delta^2 f(a_2,a_3,a_4)(c,b)$.
As the size of the formula in question becomes quite large, we shall split it.
We have
\begin{eqnarray*}
\Im f\left(\begin{array}{cc}
a_3 & b\\
0 & a_4
\end{array}\right)^{-1} 
& = & \left(\begin{array}{cc}
\Im f(a_3) & \frac{\Delta f(a_3,a_4)(b)}{2i}\\
\frac{\Delta f(a_3,a_4)(b)^*}{-2i} & f(a_4)
\end{array}\right)^{-1}\\
& = & \left(\begin{array}{cc}
e_{11} & e_{12}\\
e_{21} & e_{22}
\end{array}\right),
\end{eqnarray*}
where
\begin{eqnarray*}
e_{11}&=&\left(\Im f(a_3)-\frac{\Delta f(a_3,a_4)(b)(\Im f(a_4))^{-1}\Delta f(a_3,a_4)(b)^*}{4}\right)^{-1}\\
e_{12}&=&\left(\Im f(a_3)-\frac{\Delta f(a_3,a_4)(b)(\Im f(a_4))^{-1}\Delta f(a_3,a_4)(b)^*}{4}
\right)^{-1}\\
& & \mbox{}\times\frac{\Delta f(a_3,a_4)(b)}{-2i}(\Im f(a_4))^{-1}\\
e_{21}&=&(\Im f(a_4))^{-1}\frac{\Delta f(a_3,a_4)(b)^*}{2i}\\
& & \mbox{}\times\left(\Im f(a_3)-\frac{\Delta f(a_3,a_4)(b)(\Im f(a_4))^{-1}\Delta f(a_3,a_4)(b)^*}{4}
\right)^{-1}=e_{12}^*\\
e_{22}&=&\left(\Im f(a_4)-\frac{\Delta f(a_3,a_4)(b)^*(\Im f(a_3))^{-1}\Delta f(a_3,a_4)(b)}{4}\right)^{-1}.
\end{eqnarray*}
Thus, in the left-hand side of \eqref{est'''} preserves only one nonzero element, in the
position $22$ (lower right corner), namely
\begin{eqnarray*}
\lefteqn{\Delta f(a_2,a_3)(c)e_{11}\Delta f(a_2,a_3)(c)^*+2\Re\left(\Delta f(a_2,a_3)(c)e_{12}
\Delta^2 f(a_2,a_3,a_4)(c,b)^*\right)}\\
& & \mbox{}+\Delta^2 f(a_2,a_3,a_4)(c,b)e_{22}\Delta^2 f(a_2,a_3,a_4)(c,b)^* \\
&=&\Delta f(a_2,a_3)(c)e_{11}\Delta f(a_2,a_3)(c)^*\\
& & \mbox{}-\Im\left(
\Delta f(a_2,a_3)(c)e_{11}\Delta f(a_3,a_4)(b)(\Im f(a_4))^{-1}
\Delta^2 f(a_2,a_3,a_4)(c,b)^*\right)\\
& & \mbox{}+\Delta^2 f(a_2,a_3,a_4)(c,b)e_{22}\Delta^2 f(a_2,a_3,a_4)(c,b)^*.
\end{eqnarray*}
On the right-hand side of \eqref{est'''} we have the norm
$$
\left\|\left(\begin{array}{cc}
(\Im a_1)^{-\frac12} & 0\\
0 & (\Im a_2)^{-\frac12}
\end{array}\right)\left(\begin{array}{cc}
0 & 0\\
c & 0
\end{array}\right)\left(\begin{array}{cc}
\Im a_3 & \frac{b}{2i}\\
\left(\frac{b}{2i}\right)^* & \Im a_4
\end{array}\right)^{-\frac12}\right\|.
$$
We use the properties of C${}^*$-norms to conclude that this norm in $M_2(\mathcal A)$ in fact
equals the norm $\left\|(\Im a_2)^{-\frac12}c \left(\Im a_3-\frac14b(\Im a_4)^{-1}b^*
\right) c^*(\Im a_2)^{-\frac12}\right\|$ in $\mathcal A$.
Thus, inequality \eqref{est'''} for elements in $M_2(\mathcal A)$ translates into an
inequality of elements in $\mathcal A$ as follows:
\begin{eqnarray}
\lefteqn{\Delta f(a_2,a_3)(c)e_{11}\Delta f(a_2,a_3)(c)^*}\nonumber\\
& & \mbox{}-\Im\left(
\Delta f(a_2,a_3)(c)e_{12}\Delta f(a_3,a_4)(b)(\Im f(a_4))^{-1}
\Delta^2 f(a_2,a_3,a_4)(c,b)^*\right)\nonumber\\
& & \mbox{}+\Delta^2 f(a_2,a_3,a_4)(c,b)e_{22}\Delta^2 f(a_2,a_3,a_4)(c,b)^*\nonumber\\
& \leq & \left\|(\Im a_2)^{-\frac12}c \left(\Im a_3-\frac14b(\Im a_4)^{-1}b^*
\right)^{-1} c^*(\Im a_2)^{-\frac12}\right\|\Im f(a_2).\label{secondderiv}
\end{eqnarray}

However, for now their form 
seems to be too complicated when $j>2$, and of no significant 
use for the purposes of this paper. Since the above proposition is 
applied in this paper only for $\mathcal A=\mathcal C$, from now
on we shall eliminate the subscript from the notation of the C${}^*$-norm of $\mathcal A$.

\begin{prop}\label{Hyper}
Fix $n\in\mathbb N$, $r>0$ and $c\in H^+(M_n(\mathcal A))$. Denote 
$$
B^+_n(c,r)=\left\{a\in H^+(M_n(\mathcal A))\colon\left\|(\Im a)^{-1/2}(a-c)(\Im c)^{-1/2}\right\|\leq r
\right\}.
$$
Then $B_n^+(c,r)$ is a norm-closed norm-bounded convex subset of $H^+(M_n(\mathcal A))$ with
nonempty interior, which is bounded away from the topological boundary in the norm topology of 
$H^+(M_n(\mathcal A))$. Moreover, it is noncommutative. More precisely,
\begin{equation}\label{estimBall}
\|a\|\leq\|\Re c\|+\|\Im c\|\left(\frac{r^2+2+r\sqrt{r^2+4}}{2}+r\sqrt{\frac{r^2+2+r\sqrt{r^2+4}}{2}}\right),\ a\in B^+_n(c,r),
\end{equation}
and 
\begin{equation}\label{estimBdry}
\Im a\ge\frac{1}{2+r^2}\Im c,\quad a\in B^+_n(c,r).
\end{equation}

\end{prop}
\begin{proof}
The set $B^+_n(c,r)$ is norm-bounded: $a\in B^+_n(c,r)$ if and only if $(\Im a)^{-\frac12}(a-c)
(\Im c)^{-1}(a-c)^*(\Im a)^{-\frac12}\leq r^2\cdot1$, relation which implies $(a-c)(\Im c)^{-1}
(a-c)^*\leq r^2\|\Im a\|\cdot1$, which in its own turn implies $\left\|[(a-c)(\Im c)^{-\frac12}]
[(a-c)(\Im c)^{-\frac12}]^*\right\|\leq r^2\|\Im a\|$. Recalling that in any C${}^*$-algebra the
adjoint (star) operation is isometric and that $\|x^*x\|=\|x\|^2$, this implies that
$\left\|[(a-c)(\Im c)^{-\frac12}]^*
[(a-c)(\Im c)^{-\frac12}]\right\|\leq r^2\|\Im a\|$, which again implies 
$(\Im c)^{-\frac12}(a-c)^*(a-c)(\Im c)^{-\frac12}\leq r^2\|\Im a\|\cdot1$. Thus, repeating once again 
the above computations, we obtain
$$
\|a-c\|^2\leq r^2\|\Im a\|\|\Im c\|,\quad a\in B^+_n(c,r).
$$
Recall that $\Im x=(x-x^*)/2i$, so $\|\Im x\|\leq(\|x\|+\|x^*\|)/2=\|x\|$. Similarly,  $\|\Re x\|\leq
\|x\|.$ Applying this to $x=a-c$, we obtain
$$
\left(\|\Im a\|-\|\Im c\|\right)^2\leq\|\Im (a-c)\|^2\leq\|a-c\|^2\leq r^2\|\Im a\|\|\Im c\|
,\quad a\in B^+_n(c,r).
$$
Direct computation shows that this relation imposes 
\begin{equation}\label{estimIm}
\frac{\|\Im c\|}{2}\left(r^2+2-r\sqrt{r^2+4}\right)\leq\|\Im a\|\leq
\frac{\|\Im c\|}{2}\left(r^2+2+r\sqrt{r^2+4}\right), 
\end{equation}
for all $a\in B^+_n(c,r).$ Similarly,$\|\Re(a-c)\|^2\leq\|a-c\|^2\leq  r^2\|\Im a\|\|\Im c\|$ implies
\begin{equation}\label{estimRe}
0\leq\|\Re a\|\le\|\Re c\|+r\|\Im c\|\sqrt{\frac{r^2+2+r\sqrt{r^2+4}}{2}},\quad a\in B^+_n(c,r).
\end{equation}
Adding relations \eqref{estimIm} and \eqref{estimRe} provides the bound
$$
\|a\|\leq\|\Re c\|+\|\Im c\|
\left(\frac{r^2+2+r\sqrt{r^2+4}}{2}+r\sqrt{\frac{r^2+2+r\sqrt{r^2+4}}{2}}\right),
$$
as claimed in our remark.

Relation \eqref{estimBdry} is proved by a direct application of one of the equivalent definitions
of positivity in a von Neumann algebra and the Cauchy-Buniakovsky-Schwarz inequality in Hilbert
spaces. Let $\xi$ be an arbitrary vector in the Hilbert space $\mathcal H^n$ on which $M_n(
\mathcal A)$ acts as a von Neumann algebra. As we have seen in the proof of \eqref{estimBall}
above, $a\in B^+_n(c,r)\iff (a-c)(\Im c)^{-1}(a-c)^*\leq r^2\Im a$. This means that
$$
\left\langle (a-c)(\Im c)^{-1}(a-c)^*\xi,\xi\right\rangle\leq r^2\langle\Im a\xi,\xi\rangle.
$$
Moving $a-c$ to the right with a star and taking real and imaginary parts provides us with
\begin{eqnarray*}
\lefteqn{\left\|(\Im c)^{-\frac12}\Re(a-c)\xi\right\|_2^2+\left\|(\Im c)^{-\frac12}\Im a\xi\right\|_2^2
+\langle\Im c\xi,\xi\rangle}\\
& & \mbox{}+i\left(\left\langle(\Im c)^{-\frac12}\Re(a-c)\xi,(\Im c)^{-\frac12}\Im a\xi\right\rangle
-\overline{\left\langle(\Im c)^{-\frac12}\Re(a-c)\xi,(\Im c)^{-\frac12}\Im a\xi\right\rangle}\right)\\
&\leq &(2+r^2)\langle\Im a\xi,\xi\rangle.
\end{eqnarray*}
Second line above is simply $-2\Im\left\langle(\Im c)^{-\frac12}\Re(a-c)\xi,(\Im c)^{-\frac12}\Im 
a\xi\right\rangle$, which is clearly greater than $-2\left|
\left\langle(\Im c)^{-\frac12}\Re(a-c)\xi,(\Im c)^{-\frac12}\Im a\xi\right\rangle\right|.$ By the 
Schwarz-Cauchy inequality (applied in the second inequality below) we obtain
\begin{eqnarray*}
\langle\Im c\xi,\xi\rangle & \leq & \langle\Im c\xi,\xi\rangle+\left(
\left\|(\Im c)^{-\frac12}\Re(a-c)\xi\right\|_2-\left\|(\Im c)^{-\frac12}\Im a\xi\right\|_2\right)^2\\
& = & \langle\Im c\xi,\xi\rangle+\left\|(\Im c)^{-\frac12}\Re(a-c)\xi\right\|_2^2+\left\|(\Im 
c)^{-\frac12}\Im a\xi\right\|_2^2\\
& & \mbox{}-2\left\|(\Im c)^{-\frac12}\Re(a-c)\xi\right\|_2\left\|(\Im c)^{-\frac12}\Im a\xi\right\|_2\\
& \leq & \langle\Im c\xi,\xi\rangle+\left\|(\Im c)^{-\frac12}\Re(a-c)\xi\right\|_2^2+\left\|(\Im 
c)^{-\frac12}\Im a\xi\right\|_2^2\\
& & \mbox{}-2\left|
\left\langle(\Im c)^{-\frac12}\Re(a-c)\xi,(\Im c)^{-\frac12}\Im a\xi\right\rangle\right|\\
& \leq & \langle\Im c\xi,\xi\rangle+\left\|(\Im c)^{-\frac12}\Re(a-c)\xi\right\|_2^2+\left\|(\Im 
c)^{-\frac12}\Im a\xi\right\|_2^2\\
& & \mbox{}-2\Im\left\langle(\Im c)^{-\frac12}\Re(a-c)\xi,(\Im c)^{-\frac12}\Im 
a\xi\right\rangle\\
&\leq &(2+r^2)\langle\Im a\xi,\xi\rangle.
\end{eqnarray*}
Since this is true for all vectors $\xi\in\mathcal H^n$, we obtain $\Im c\leq(2+r^2)\Im a$, implying
\eqref{estimBdry}. 

That $B^+_n(c,r)$ is closed in norm follows even easier: if $a_m\in B^+_n(c,r)$ and $\lim_{m\to\infty}
\|a_m-a\|=0$, then $\lim_{m\to\infty}\|a_m^*-a^*\|=0$, and thus $\lim_{m\to\infty}
\|\Im a_m-\Im a\|=0$. This also implies that $\Im a\ge\frac{1}{2+r^2}\Im c>0$, so that, by 
analytic functional calculus, $\lim_{m\to\infty}\left\|(\Im a_m)^{-\frac12}-(\Im a)^{-\frac12}\right\|=0$.
A few succesive applications of the product-norm inequalities in C${}^*$-algebras provides 
\begin{eqnarray*}
\left\|(\Im a)^{-\frac12}(a-c)(\Im c)^{-\frac12}\right\|&\leq&\left\|(\Im a)^{-\frac12}\right\|\|a_m-a\|
+\left\|(\Im a_m)^{-\frac12}-(\Im a)^{-\frac12}\right\|\\
& & \mbox{}\times\left\|(a_m-c)(\Im c)^{-\frac12}\right\|+\left\|(\Im a_m)^{-\frac12}(a_m-c)
(\Im c)^{-\frac12}\right\|.
\end{eqnarray*}
First and second right-hand terms converge to zero as $m\to\infty$, and the last is no more than $r$. 
Thus, $\left\|(\Im a)^{-\frac12}(a-c)(\Im c)^{-\frac12}\right\|\leq r$, which implies $a\in B^+_n(c,r)$.

Midpoint convexity of $B^+_n(c,r)$ follows easily from a direct computation: let
$a_1,a_2\in B_n^+(c,r)$. We show that $(a_1+a_2)/2$ is in $B_n^+(c,r)$.
As in \eqref{est}, this is equivalent to showing that 
$$
\left(\Im\frac{a_1+a_2}{2}\right)^{-\frac12}\left(\frac{a_1+a_2}{2}-c\right)(\Im c)^{-1}
\left(\frac{a_1+a_2}{2}-c\right)^*\left(\Im\frac{a_1+a_2}{2}\right)^{-\frac12}\leq r^2\cdot1,
$$
which is in its own turn equivalent to 
\begin{equation}\label{conv}
\left(\frac{a_1-c}{2}+\frac{a_2-c}{2}\right)(\Im c)^{-1}\left(\frac{a_1-c}{2}+\frac{a_2-c}{2}\right)^*\leq
\frac{r^2}{2}\Im(a_1+a_2).
\end{equation}
However, adding the inequalities $(a_1-c)(\Im c)^{-1}(a_1-c)^*\leq r^2\Im a_1$ and
$(a_2-c)(\Im c)^{-1}(a_2-c)^*\leq r^2\Im a_2$ (assumed to be true by hypothesis) and dividing by 2, 
we obtain
$$
\frac12((a_1-c)(\Im c)^{-1}(a_1-c)^*+\frac12(a_2-c)(\Im c)^{-1}(a_2-c)^*\leq\frac{r^2}{2}\Im(a_1+a_2).
$$
Thus, our statement is proved if we show that the left-hand term of \eqref{conv} is less than or
equal to the left-hand term of the inequality above. Expanding the left-hand of \eqref{conv} and
subtracting from the one above yields
\begin{eqnarray*}
 &  & \frac12(a_1-c)(\Im c)^{-1}(a_1-c)^*+\frac12(a_2-c)(\Im c)^{-1}(a_2-c)^*\\
& & \mbox{}-\frac14(a_1-c)(\Im c)^{-1}(a_1-c)^*-\frac14(a_2-c)(\Im c)^{-1}(a_2-c)^*\\
& & \mbox{}-\frac14(a_1-c)(\Im c)^{-1}(a_2-c)^*-\frac14(a_2-c)(\Im c)^{-1}(a_1-c)^*\\
& = & \frac14\left[(a_1-c)(\Im c)^{-1}(a_1-c-a_2+c)^*+(a_2-c)(\Im c)^{-1}(a_2-c-a_1+c)^*\right]\\
& = & \frac14(a_1-c-a_2+c)(\Im c)^{-1}(a_1-a_2)^*=\frac14(a_1-a_2)(\Im c)^{-1}(a_1-a_2)^*\ge0.
\end{eqnarray*}
Since $B^+_n(c,r)$ is midpoint convex and closed, it is convex.

To conclude, observe that all the above computations hold if $c\in H^+(M_n(\mathcal A))$ is replaced
by $c\otimes 1_p\in H^+(M_{np}(\mathcal A))$. Indeed, one only needs to observe that taking
imaginary part, inverse and root, as well as multiplication, respect direct sums. Since $\|a\oplus b\|=
\max\{\|a\|,\|b\|\}$, we're done. Estimates \eqref{estimBall} and \eqref{estimBdry}
hold on the amplifications of $c$ to any  $c\otimes 1_p$, $p\in\mathbb N$, with the same constants.
\end{proof}

The following lemma will be useful when applying
Proposition \ref{prop:3.1} to the proof of the main result (compare with the method used in
\cite[Remark 2.5]{BPV1}).

\begin{lemma}\label{lem:3.2}
Assume that $f$ is a noncommutative self-map of the noncommutative
upper half-plane of $\mathcal A$. Let $v_1,v_2>0$ in $\mathcal A$. If 
$$
{\rm wo}\text{-}\lim_{y\downarrow0}\frac{\Im f(\alpha+iyv_j)}{y}=c_j\in
\mathcal A,\quad j\in\{1,2\},
$$
exist, then the set of limit points of $\Delta f(\alpha+zv_1,
\alpha+\zeta v_2)(w)$ as $z,\zeta\to 0$ nontangentially is bounded uniformly in norm as $w$ 
varies in the unit ball of $\mathcal A$.
\end{lemma}
\begin{proof}
By Proposition \ref{prop:3.1}, 
\begin{eqnarray*}
\lefteqn{\left\|\left(\Im f(\alpha+zv_1)\right)^{-\frac12}
\Delta f(\alpha+zv_1,\alpha+\zeta v_2)(w)
\left(\Im f(\alpha+\zeta v_2)\right)^{-\frac12}\right\|}\\
&\leq&\left\|(\Im zv_1)^{-\frac12}w(\Im\zeta v_2)^{-\frac12}\right\|.
\quad\quad\quad\quad\quad\quad\quad\quad\quad\quad\quad\quad\quad\quad
\end{eqnarray*}
Multiplying by $(\Im z\Im\zeta)^{1/2}$ we obtain
$$
\left\|\left[\frac{\Im f(\alpha+zv_1)}{\Im z}\right]^{-\frac12}
\Delta f(\alpha+zv_1,\alpha+\zeta v_2)(w)\left[\frac{\Im f(\alpha+\zeta v_2)}{\Im\zeta}\right]^{-\frac12}
\right\|\leq\left\|v_1^{-\frac12}wv_2^{-\frac12}\right\|.
$$
Let $\varepsilon\ge0$ be fixed, and denote $f_\varepsilon(a)=
f(a)+\varepsilon a$, i.e. $f_\varepsilon=f+\varepsilon{\rm Id}$.
Since $\text{Id}$ is completely positive, $f_\varepsilon$ is still
a noncommutative self-map of the noncommutative upper half-plane
of $\mathcal A$, so that
\begin{eqnarray*}
\lefteqn{\left
\|\left(\frac{\Im f(\alpha+zv_1)}{\Im z}+\varepsilon v_1\right)^{-\frac12}
(\Delta f(\alpha+zv_1,\alpha+\zeta v_2)(w)+\varepsilon w)\right.\times}\\
& & \left.\left(\frac{\Im f(\alpha+\zeta v_2)}{\Im\zeta}+\varepsilon v_2
\right)^{-\frac12}\right\|\\
& \leq & \left\|v_1^{-\frac12}wv_2^{-\frac12}\right\|.\quad\quad\quad
\quad\quad\quad\quad\quad\quad\quad\quad\quad\quad\quad\quad\quad\quad
\quad\quad\quad\quad\quad\quad\quad\quad
\end{eqnarray*}
For simplicity, we denote $A_1(\Im z,\varepsilon)=
\frac{\Im f(\alpha+zv_1)}{\Im z}+\varepsilon v_1$, $A_2(\Im\zeta,\varepsilon)=
\frac{\Im f(\alpha+\zeta v_2)}{\Im\zeta}+\varepsilon v_2$, $W(z,\zeta,\varepsilon)=
\Delta f(\alpha+zv_1,\alpha+\zeta v_2)(w)+\varepsilon w$, and $K=\left\|
v_1^{-\frac12}wv_2^{-\frac12}\right\|^2$. As noted in \eqref{est}, and following the same
procedure as in the proof fo the previous proposition, the above is equivalent to
$$
A_2(\Im\zeta,\varepsilon)^{-\frac12}W(z,\zeta,\varepsilon)^*
A_1(\Im z,\varepsilon)^{-1}W(z,\zeta,\varepsilon)A_2(\Im\zeta,\varepsilon)^{-\frac12}
\leq K1.
$$
As $A_j(\cdot,\varepsilon)\ge\varepsilon1$, we obtain by the same methods as in the proof of
Proposition \ref{Hyper} that
$$
\|W(z,\zeta,\varepsilon)\|^2\leq K\|A_1(\Im z,\varepsilon)\|\|A_2(\Im\zeta,\varepsilon)\|.
$$
Let $\mathcal H$ be the Hilbert space on which $\mathcal A$ acts
as a von Neumann algebra. By our hypothesis, 
$\lim_{y\downarrow 0}\frac{\langle\Im f(\alpha+iyv_j)\xi,\xi\rangle}{y}=
\lim_{y\downarrow 0}\left\|\left(\frac{\Im f(\alpha+iyv_j)}{y}
\right)^\frac12\xi\right\|_2^2$
exist and equal $\langle c_j\xi,\xi\rangle$, finite for any 
$\xi\in\mathcal H$. Thus, the family 
$\left\{\left\|\frac{\left(\Im f(\alpha+iyv_j)\right)^{1/2}}{\sqrt{y}}
\xi\right\|_2\colon y\in(0,1)\right\}$ is bounded for any 
$\xi\in\mathcal H$. By the Banach-Steinhaus Theorem and the 
positivity of the operators $\frac{\Im f(\alpha+iyv_j)}{{y}}$, it follows 
that $\left\{\left\|\frac{\Im f(\alpha+iyv_j)}{{y}}\right\|\colon 
y\in(0,1)\right\}$ is a bounded set. Moreover, as it will be seen in
the proof of Theorem \ref{JC}, if $z$ tends to zero nontangentially and 
$\lim_{y\downarrow 0}\frac{\langle\Im f(\alpha+iyv_j)\xi,\xi\rangle}{y}$ is finite,
then $\left\{\frac{\langle\Im f(\alpha+\Im zv_j)\xi,\xi\rangle}{\Im z}\colon|z|<1,z\in\Gamma\right\}$
stays bounded for any closed cone $\Gamma\subset\mathbb C^+\cup\{0\}$. 
A bound for $c_j$ is $\|c_j\|\leq\limsup_{y\to0}\left\|\frac{\Im f
(\alpha+iyv_j)}{{y}}\right\|$. Thus, $\{\|W(z,\zeta,\varepsilon)\|\colon
z,\zeta\in\Gamma,|z|,|\zeta|<1\}$ is bounded for any closed cone 
$\Gamma\subset\mathbb C^+$ with vertex at zero. The lemma follows 
by letting $\varepsilon\downarrow0$.
\end{proof}

We note that the bounds depend exclusively on $c_j,v_j (j=1,2),w$. Moreover, the dependence can
be bounded (at most) linearly in terms of $\|w\|,\|v_1\|,\|v_2\|,\|v_1^{-1}\|$ and $\|v_2^{-1}\|$.

For the sake of completeness, let us use the results of Proposition \ref{prop:3.1} to give a short, 
elementary proof of Theorem \ref{JC}.

\begin{proof}[Proof of Theorem \ref{JC}]
Assume equation \eqref{3} holds. By Proposition \ref{prop:3.1},
$$
\left|\frac{f(z)-f(z')}{\sqrt{\Im f(z)\Im f(z')}}\right|\leq
\left|\frac{z-z'}{\sqrt{\Im z\Im z'}}\right|,\quad z,z'\in
\mathbb C^+.
$$
This is equivalent to
\begin{equation}\label{9}
\left|\frac{f(z)-f(z')}{z-z'}\right|^2\leq
\left|\frac{\Im f(z)\Im f(z')}{\Im z\Im z'}\right|,\quad z,z'\in
\mathbb C^+, z\neq z'.
\end{equation}
Consider a sequence $\{z_n'\}_{n\in\mathbb N}\subset\mathbb C^+$
converging to $\alpha$ such that $\lim_{n\to\infty}\frac{\Im f(z_n')}{
\Im z_n'}=c$. Clearly $\Im f(z_n')\to0$ as $n\to\infty$, and $\{\Re
f(z_n')\}_{n\in\mathbb N}$ is a bounded sequence in $\mathbb R$.
Moreover, if $\{z_n\}_{n\in\mathbb N}$ and $\{z_n'\}_{n\in\mathbb N}$
are two arbitrary sequences converging to $\alpha$ along which 
$\Im f(z)/\Im z$ stays bounded, then $\{\Re(f(z_n)-f(z_n'))\}_{
n\in\mathbb N}$ converges to zero. This implies that $\lim_{n\to\infty}
f(z_n)$ exists for any sequence $\{z_n\}_{n\in\mathbb N}$ such that
$\{\Im f(z_n)/\Im z_n\}_{n\in\mathbb N}$ is bounded and $\lim_{n\to
\infty}z_n=\alpha$. We agree to call this limit $f(\alpha)$. Taking
limit along $z_n'$ in \eqref{9} we obtain
$$
\left|\frac{f(z)-f(\alpha)}{z-\alpha}\right|^2\leq
c\frac{\Im f(z)}{\Im z},\quad z\in\mathbb C^+.
$$
Fix an $M\in[0,+\infty)$. Let $D_M=\{z\in\mathbb C^+\colon
|\Re z-\alpha|\leq M\Im z\}$. For any $z\in D_M$, this implies
\begin{eqnarray*}
\left(\Re f(z)-f(\alpha)\right)^2 & \leq & c\frac{\Im f(z)}{\Im z}|z-
\alpha|^2-\left(\Im f(z)\right)^2\\
& = & \Im f(z)\left(\frac{c|z-\alpha|^2}{\Im z}-\Im f(z)\right)\\
& = & \Im f(z)\left(c\Im z\frac{|\Re z-\alpha|^2}{(\Im z)^2}
+c\Im z-\Im f(z)\right)\\
& \leq & \Im f(z)\left(c(M^2+1)\Im z-\Im f(z)\right).
\end{eqnarray*}
We conclude that $\Im f(z)/\Im z\leq c(M^2+1)$ for all $z\in D_M$
and thus
$\displaystyle\lim_{\stackrel{z\longrightarrow0}{{\sphericalangle}}}
f(z)=f(\alpha)$. Moreover, for $M=0$ (i.e. $z$ of the form $\alpha+iy$)
we have $c\geq\Im f(\alpha+iy)/y$, which together with the 
definition of $c$ implies $\lim_{y\downarrow0}\frac{\Im f
(\alpha+iy)}{y}=c$, so that 
$$
\frac{\left(\Re f(\alpha+iy)-f(\alpha)\right)^2}{y^2}\leq \frac{\Im f(\alpha+iy)}{y}
\left(c-\frac{\Im f(\alpha+iy)}{y}\right)\to0\quad\text{as }y\downarrow0.
$$
These two facts imply, via direct computation, that
$\lim_{y\downarrow0}\frac{f(\alpha+iy)-f(\alpha)}{iy}=c$. Since
$$
\left|\frac{f(z)-f(\alpha)}{z-\alpha}\right|^2\leq
c\frac{\Im f(z)}{\Im z}\leq c^2(M^2+1),\quad z\in D_M, M\ge0,
$$
it follows straightforwardly that 
$$
\lim_{\stackrel{z\longrightarrow0}{{\sphericalangle}}}
\frac{f(z)-f(\alpha)}{z-\alpha}=c
$$
(see for example \cite[Exercise 5, Chapter I]{garnett}).

Considering the classical definition of the derivative, the above 
directly implies that $\lim_{y\downarrow0}f'(\alpha+iy)=c.$
Relation \eqref{9} implies that $|f'(z)|\leq c(M^2+1)$ for $z\in
D_M$, so, by the same \cite[Exercise 5, Chapter I]{garnett},
$\displaystyle\lim_{\stackrel{z\longrightarrow0}{{\sphericalangle}}}
f'(z)=c$. This proves (1).

To prove (2), simply observe that 
$$
\left|\frac{f(\alpha+iy)-f(\alpha)}{iy}\right|^2=
\left|\frac{(\Re f(\alpha+iy)-f(\alpha))^2+(\Im f(\alpha+iy))^2}{y^2}
\right|\ge\frac{(\Im f(\alpha+iy))^2}{y^2},
$$
so that $\liminf_{z\to\alpha}\frac{\Im f(z)}{\Im z}<\infty$.
Part (2) follows now from part (1).

To prove part (3), we apply the classical mean value theorem 
to bound $\Im f(\alpha+iy)/y$. The result follows then from part (1).
\end{proof}
We feel it necessary to reiterate that no claim to 
novelty is made for this proof, and we chose to write it down
here for the sake of making the paper more self-contained.

%%%%%%%%%%%%%%%%%%%%%%%%%%%%%%%%%%%%%%%%%%%%%%%%%%%%%%%%%%%%%%%%%%%%%%%
%%%%%%%%%%%%%%%%%%%%%%%%%%%%%%%%%%%%%%%%%%%%%%%%%%%%%%%%%%%%%%%%%%%%%%%

\section{Proof of the main result}

In this section we prove Theorem \ref{Main}. The proof makes use
quite often of the results, and sometimes of the proof, of
Theorem \ref{JC}. For the sake of simplicity, we will isolate some 
elements of the proof in separate lemmas.

\begin{proof}[Proof of Theorem \ref{Main}]
For any $n\in\mathbb N$ and any state $\varphi$ on $M_n(\mathcal A)$, 
$z\mapsto\varphi(f(\alpha+zv))$ is a self-map of $\mathbb C^+$
whenever $\alpha$ is selfadjoint and $v>0$ in $M_n(\mathcal A)$.
Thus, Theorem \ref{JC} applies to it. In particular, if $\mathcal H$ is 
the Hilbert space on which the von Neumann algebra $\mathcal A$ acts,
the above holds for the vector state corresponding to any $\xi\in
\oplus_{j=1}^n\mathcal H$ of $L^2$-norm equal to one.
For $n=1$, our hypothesis guarantees that 
$\liminf_{z\to0}\frac{\langle\Im f(\alpha+zv)\xi,\xi\rangle}{\Im z}$
is finite. Item (1) of Theorem \ref{JC} guarantees that
$\lim_{y\downarrow 0}\frac{\langle\Im f(\alpha+iyv)\xi,\xi\rangle}{y}=
\lim_{y\downarrow 0}\left\|\left(\frac{\Im f(\alpha+iyv)}{y}
\right)^\frac12\xi\right\|_2^2$ exists and equals the above 
$\liminf$, hence it is finite for any $\xi\in
\mathcal H$. As in the proof of Lemma \ref{lem:3.2},
the Banach-Steinhaus Theorem and the 
positivity of the operators $\frac{\Im f(\alpha+iyv)}{{y}}$ guarantee
that $\left\{\left\|\frac{\Im f(\alpha+iyv)}{{y}}\right\|\colon 
y\in(0,1)\right\}$ is a bounded set. Moreover, the existence of 
the limits
$\lim_{y\downarrow 0}\frac{\langle\Im f(\alpha+iyv)\xi,\xi\rangle}{y}$
for all $\xi\in\mathcal H$ implies, via polarization, the existence
of 
$$
\lim_{y\downarrow 0}\frac{\langle\Im f(\alpha+iyv)\xi,\eta\rangle}{y},
\quad \xi,\eta\in\mathcal H.
$$
We conclude the existence of a bounded operator $0\le c=c(v)\in\mathcal 
A$ such that 
$$
\lim_{y\downarrow 0}\frac{\langle\Im f(\alpha+iyv)\xi,\eta\rangle}{y}=
\langle c\xi,\eta\rangle,\quad \xi,\eta\in\mathcal H.
$$
The bound for $c$ is $\|c\|\leq\limsup_{y\to0}\left\|\frac{\Im f
(\alpha+iyv)}{{y}}\right\|$. On the other hand, as seen in the proof of Theorem \ref{JC},
$\Im\langle f(\alpha+iyv)\xi,\xi\rangle\leq y\langle c\xi,\xi\rangle$ for all
$y>0$. Since $f$ takes values in $H^+(\mathcal A)$, applying this relation to
$y=1$ guaranteres that $c>0$. Now it follows easily that 
$\lim_{y\downarrow0}\left\|\left(\frac{\Im f(\alpha+iyv)}{y}-c\right)\xi\right\|=0$
for any $\xi\in\mathcal H$.

We show next that the limit
$\lim_{y\downarrow 0}f(\alpha+iyv)=f(\alpha)$
exists in $\mathcal A$ (i.e. does not depend on $v$) and is 
selfadjoint. Indeed, consider again any state $\varphi$ on $\mathcal A
$ and define $z\mapsto\varphi(f(\alpha+zv))$. We have seen that this is 
a self-map of $\mathbb C^+$ to which Theorem \ref{JC} applies.
Thus, there exists a number $k=k(\varphi,\alpha,v)
\in\mathbb R$ such that $\displaystyle\lim_{\stackrel{z\longrightarrow
0}{\sphericalangle}}\varphi(f(\alpha+zv))=k.$
We recall the estimate from Proposition \ref{prop:3.1} 
$$
\left|
\frac{\varphi(f(\alpha+zv))-\varphi(f(\alpha+z'v))}{z-z'}\right|^2
\leq\frac{\varphi(\Im f(\alpha+zv))\varphi(\Im f(\alpha+z'v))}{\Im z
\Im z'}.
$$
In this estimate we take $z'=i$ and let $z=iy$ tend to zero. We obtain
$$
\left|k(\varphi,\alpha,v)-\varphi(f(\alpha+iv))\right|^2
\leq\varphi(c)\varphi(\Im f(\alpha+iv)).
$$
Obviously, $|\varphi(f(\alpha+iv))|\leq\|f(\alpha+iv)\|,$
a value independent of $\varphi$. Thus,
$$
|k(\varphi,\alpha,v)|\leq\|f(\alpha+iv)\|+\sqrt{\|c\|
\|\Im f(\alpha+iv)\|},
$$
for any state $\varphi$ on $\mathcal A$. By applying 
as before this result to vector states and using polarization, 
we find an operator $f_v(\alpha)\in\mathcal A$ such that
$$
\langle f_v(\alpha)\xi,\eta\rangle=\lim_{y\downarrow 0}
\langle f(\alpha+iv)\xi,\eta\rangle,\quad\xi,\eta\in\mathcal H.
$$
Since $\|x\|=\sup\{|\varphi(x)|\colon\varphi\text{ state on }
\mathcal A\}$, the estimate 
$$
\|f_v(\alpha)\|\leq4\left(\|f(\alpha+iv)\|+\sqrt{\|c\|
\|\Im f(\alpha+iv)\|}\right)
$$
holds. Since for any state $\varphi$, $k(\varphi,\alpha,v)=
\lim_{y\downarrow0}\varphi(f(\alpha+iyv))\in\mathbb R$, it follows that
$f_v(\alpha)=f_v(\alpha)^*$. The fact that $f_v(\alpha)$
does not depend on $v$ follows from
Proposition \ref{prop:3.1} and Lemma \ref{lem:3.2}: indeed,
\begin{eqnarray*}
\lefteqn{
\left\|\left(\Im f(\alpha+iy_1v)\right)^{-\frac12}
(f(\alpha+iy_1v)-f(\alpha+iy_21))\left(\Im f(\alpha+iy_21)\right)^{-\frac12}
\right\|}\\
& \leq &\left\|\left(y_1v\right)^{-\frac12}
(iy_1v-iy_21)\left(y_21)\right)^{-\frac12}
\right\|\quad\quad\quad\quad\quad\quad\quad\quad
\end{eqnarray*}
is equivalent to 
\begin{eqnarray*}
\lefteqn{
\left\|\left(\frac{\Im f(\alpha+iy_1v)}{y_1}\right)^{-\frac12}
(f(\alpha+iy_1v)-f(\alpha+iy_21))\left(\frac{\Im f(\alpha+iy_21)}{y_2}\right)^{-\frac12}
\right\|}\\
& \leq & \left\|v^{-\frac12}\right\|
\left\|y_1v-y_21\right\|.\quad\quad\quad\quad\quad\quad\quad\quad\quad\quad\quad\quad\quad\quad\quad\quad\quad
\end{eqnarray*}
We obtain as in the proof of Lemma \ref{lem:3.2}
\begin{eqnarray}
\lefteqn{
\|f(\alpha+iy_1v)-f(\alpha+iy_21)\|}\nonumber\\
& \leq & \left\|v^{-\frac12}\right\|
\left\|y_1v-y_21\right\|\sqrt{\left\|\frac{\Im f(\alpha+iy_1v)}{y_1}\right\|
\left\|\frac{\Im f(\alpha+iy_21)}{y_2}\right\|}.\label{10}
\end{eqnarray}
The two factors under the square root are bounded by hypothesis. Thus, we conclude.

\begin{remark}
This result is similar to results in \cite{AMcY,Fan,Wlo}.
We observe that this essentially improves the convergence 
to norm convergence, without requiring norm convergence 
in formula \eqref{5}.
\end{remark}

In the classical Julia-Carath\'eodory Theorem, we noted also that 
$(\Re f(\alpha+iy)-f(\alpha))/y\to0$ as $y\searrow0$. A similar result
holds for general noncommuttive functions. Indeed, using relation
\eqref{est''} with $a=\alpha+iyv,c=\alpha+iy'v$, $b=a-c$ we obtain
$$
\left(f(\alpha+iyv)-f(\alpha+iy'v)\right)^*\left(\Im f(\alpha+iyv)\right)^{-1}
\left(f(\alpha+iyv)-f(\alpha+iy'v)\right)
$$
$$
\leq\frac{(y-y')^2}{yy'}\Im f(\alpha+iy'v).
$$
Letting $y'\searrow0$ we obtain (with the notation from the statement of Theorem \ref{Main})
$$
\left(f(\alpha+iyv)-f(\alpha)\right)^*\left(\Im f(\alpha+iyv)\right)^{-1}
\left(f(\alpha+iyv)-f(\alpha)\right)\leq yc(v).
$$
Recalling that $f(\alpha)=f(\alpha)^*$ we conclude that 
$$
\left(\Re f(\alpha+iyv)-f(\alpha)\right)\left(\Im f(\alpha+iyv)\right)^{-1}
\left(\Re f(\alpha+iyv)-f(\alpha)\right)\leq yc(v)-\Im f(\alpha+iyv).
$$
We divide by $y$ and let $y\searrow0$ to conclude that
\begin{equation}\label{Re0}
0\leq\lim_{y\downarrow0}\frac{\Re f(\alpha+iyv)-f(\alpha)}{y}\left(\frac{\Im f(\alpha+iyv)}{y}\right)^{-1}
\frac{\Re f(\alpha+iyv)-f(\alpha)}{y}\leq0.
\end{equation}
The invertibility of $c(v)$ guarantees that $\lim_{y\downarrow0}\frac{\Re f(\alpha+iyv)-f(\alpha)}{y}=0$
in the so-topology. Thus, 
${\displaystyle\lim_{\stackrel{ z\longrightarrow0}{{
\sphericalangle}}}}\frac{\Re f(\alpha+zv)-f(\alpha)}{\Im z}=0$.

In order to extend the above result to all levels $n$,
we need the following lemma.

\begin{lemma}\label{lem:4.2}
Let $f$ be as in Theorem \ref{Main}. Fix $\alpha=\alpha^*\in\mathcal A$, $v_1,v_2>0$ in 
$\mathcal A$, and $b\in\mathcal A$ of norm $\|b\|^2\cdot1<v_2\|v_1^{-1}\|^{-1}$.
Then 
$$
\left\{\frac1y\left\|f\left(\begin{array}{cc}
\alpha+iyv_1 & iyb \\
0 & \alpha+iyv_2
\end{array}\right)-
f\left(\begin{array}{cc}
\alpha+iyv_1 & \frac{iyb}{2} \\
\frac{iyb^*}{2} & \alpha+iyv_2
\end{array}\right)
\right\|\colon y\in(0,1)\right\}
$$
is bounded
\end{lemma}

\begin{proof}
Observe that $\|b\|^21<4\|v_1^{-1}\|^{-1}v_2$ implies $\Im\left(\begin{array}{cc}
\alpha+iyv_1 & iyb \\
0 & \alpha+iyv_2
\end{array}\right)>0$ for all $y>0$. 
%Need $b^*v_1^{-1}b<4v_2$, which is implied by $b^*v_1^{-1}b\leq\|b^*v_1^{-1}b\|=\|v_1^{-1/2}bb^*v_1^{-1/2}\|\leq\|bb^*\|\|v^{-1}\|<4v_2\iff \|b\|^2<4v_2\|v^{-1}\|^{-1}$.
We use the same trick as in Lemma \ref{lem:3.2}. For simplicity,
denote 
$$
\mathfrak D=f\left(\begin{array}{cc}
\alpha+iyv_1 & iyb \\
0 & \alpha+iyv_2
\end{array}\right)-
f\left(\begin{array}{cc}
\alpha+iyv_1 & \frac{iyb}{2} \\
\frac{iyb^*}{2} & \alpha+iyv_2
\end{array}\right).
$$
Proposition \ref{prop:3.1} (in the guise of inequality \eqref{est}) applied to $a$ and $c$ equal to the two 
arguments of the function $f$ in the formula of $\mathfrak D$ above give
\begin{eqnarray*}
\lefteqn{
\Im f\left(\left(\begin{array}{cc}
\alpha & 0 \\
0 & \alpha
\end{array}\right)+iy
\left(\begin{array}{cc}
v_1 & {b} \\
0 & v_2
\end{array}\right)\right)^{-\frac12}
\mathfrak D^*\Im f\left(\left(\begin{array}{cc}
\alpha & 0 \\
0 & \alpha
\end{array}\right)+iy
\left(\begin{array}{cc}
v_1 & \frac{b}{2} \\
\frac{b^*}{2} & v_2
\end{array}\right)\right)^{-1}}\\
& & \mbox{}\times
\mathfrak D
\Im f\left(\left(\begin{array}{cc}
\alpha & 0 \\
0 & \alpha
\end{array}\right)+iy
\left(\begin{array}{cc}
v_1 & {b} \\
0 & v_2
\end{array}\right)\right)^{-\frac12}\\
& \leq & \left\|\left(\begin{array}{cc}
yv_1 & \frac{yb}{2} \\
\frac{yb^*}{2} & yv_2
\end{array}\right)^{-\frac12}
\left(\begin{array}{cc}
0 & \frac{-iyb}{2} \\
\frac{iyb^*}{2} & 0
\end{array}\right)
\left(\begin{array}{cc}
yv_1 & \frac{yb}{2} \\
\frac{yb^*}{2} & yv_2
\end{array}\right)^{-\frac12}
\right\|^2\cdot1_{M_2(\mathcal A)},
\quad\quad\quad\quad
\end{eqnarray*}
for all $y>0$ (we have kept the $y$'s on the right hand side for transparency of the method).
As in the proof of Lemma \ref{lem:3.2}, we ``multiply out'' the imaginary parts of $f$ on the left to
obtain
\begin{eqnarray*}
\lefteqn{\mathfrak D\mathfrak D^*\leq\|\mathfrak D\|^21}\\
& \leq & \left\|y\Im f\left(\left(\begin{array}{cc}
\alpha & 0 \\
0 & \alpha
\end{array}\right)+iy
\left(\begin{array}{cc}
v_1 & \frac{b}{2} \\
\frac{b^*}{2} & v_2
\end{array}\right)\right)\right\|
\left\|\left(\begin{array}{cc}
v_1 & \frac{b}{2} \\
\frac{b^*}{2} & v_2
\end{array}\right)^{-\frac12}
\left(\begin{array}{cc}
0 & \frac{-ib}{2} \\
\frac{ib^*}{2} & 0
\end{array}\right)\right.\\
& & \mbox{}\times\left.
\left(\begin{array}{cc}
v_1 & \frac{b}{2} \\
\frac{b^*}{2} & v_2
\end{array}\right)^{-\frac12}
\right\|^2
\left\|\frac1y\Im f\left(\left(\begin{array}{cc}
\alpha & 0 \\
0 & \alpha
\end{array}\right)+iy
\left(\begin{array}{cc}
v_1 & {b} \\
0 & v_2
\end{array}\right)\right)
\right\|\cdot1.
\end{eqnarray*}
The last factor on the right hand side is bounded by the hypothesis, formula 
\eqref{FDQ}, Lemma \ref{lem:3.2} and the above arguments. The first factor needs 
not apriori tend to zero, but it is clearly bounded. 
% by $\left\|\Im f\left(\left(\begin{array}{cc}\alpha & 0 \\ 0 & \alpha \end{array}\right)+ i\left(\begin{array}{cc} v_1 & \frac{b}{2} \\  \frac{b^*}{2} & v_2 \end{array}\right)\right)\right\|$
However, if this factor is nonzero, 
consider $\mathcal H$ to be the Hilbert space on which $\mathcal A$ acts as a 
von Neumann algebra. Then there exists a vector $\xi\in\mathcal H^2$ of norm 
one such that 
$\lim_{y\downarrow0}y\varphi_\xi\left(
\Im f\left(\left(\begin{array}{cc}
\alpha & 0 \\
0 & \alpha
\end{array}\right)+iy
\left(\begin{array}{cc}
v_1 & \frac{b}{2} \\
\frac{b^*}{2} & v_2
\end{array}\right)\right)\right)$ exists and belongs to $(0,+\infty)$, so that necessarily 
$$
\left\|\Im f\left(\left(\begin{array}{cc}
\alpha & 0 \\
0 & \alpha
\end{array}\right)+iy
\left(\begin{array}{cc}
v_1 & \frac{b}{2} \\
\frac{b^*}{2} & v_2
\end{array}\right)\right)\right\|\to+\infty,\quad y\to0.
$$
(Recall that we have denoted by $\varphi_\xi$ the vector state corresponding to $\xi$: $\varphi_\xi(a)
=\langle a\xi,\xi\rangle$.) But then
$2\|\Im\mathfrak D\|=\|\mathfrak D-\mathfrak D^*\|\le2\|\mathfrak D\|$ is unbounded as $y$
tends to zero, so that
\begin{eqnarray}
\lefteqn{\|\Im\mathfrak D\|^2\leq\|\mathfrak D\|^2}\nonumber\\
& \leq & \left\|y\Im f\left(\left(\begin{array}{cc}
\alpha & 0 \\
0 & \alpha
\end{array}\right)+iy
\left(\begin{array}{cc}
v_1 & \frac{b}{2} \\
\frac{b^*}{2} & v_2
\end{array}\right)\right)\right\|
\left\|\left(\begin{array}{cc}
v_1 & \frac{b}{2} \\
\frac{b^*}{2} & v_2
\end{array}\right)^{-\frac12}\nonumber
\left(\begin{array}{cc}
0 & \frac{-ib}{2} \\
\frac{ib^*}{2} & 0
\end{array}\right)\right.\\
& & \mbox{}\times\left.
\left(\begin{array}{cc}
v_1 & \frac{b}{2} \\
\frac{b^*}{2} & v_2
\end{array}\right)^{-\frac12}
\right\|^2
\left\|\frac1y\Im f\left(\left(\begin{array}{cc}
\alpha & 0 \\
0 & \alpha
\end{array}\right)+iy
\left(\begin{array}{cc}
v_1 & {b} \\
0 & v_2
\end{array}\right)\right)
\right\|,\label{12}
\end{eqnarray}
making the right hand side unbounded, a contradiction. Thus, 
$$
\lim_{y\to0}\left\|y\Im f\left(\left(\begin{array}{cc}
\alpha & 0 \\
0 & \alpha
\end{array}\right)+iy
\left(\begin{array}{cc}
v_1 & \frac{b}{2} \\
\frac{b^*}{2} & v_2
\end{array}\right)\right)\right\|=0,
$$
so, by a second application of inequality \eqref{12},
$$
\lim_{y\to0}\left\|\Im f\left(\left(\begin{array}{cc}
\alpha & 0 \\
0 & \alpha
\end{array}\right)+iy
\left(\begin{array}{cc}
v_1 & \frac{b}{2} \\
\frac{b^*}{2} & v_2
\end{array}\right)\right)\right\|=0.
$$
However, more can be concluded from \eqref{12}: dividing by $y^2$, one obtains
\begin{eqnarray*}
\lefteqn{\frac{\|\Im \mathfrak D\|^2}{y^2}=\left\|\frac1y\Im f\left(\begin{array}{cc}
\alpha+iyv_1 & iyb \\
0 & \alpha+iyv_2
\end{array}\right)-
\frac1y\Im f\left(\begin{array}{cc}
\alpha+iyv_1 & \frac{iyb}{2} \\
\frac{iyb^*}{2} & \alpha+iyv_2
\end{array}\right)\right\|^2}\\
& \leq & \left\|\frac1y\Im f\left(\left(\begin{array}{cc}
\alpha & 0 \\
0 & \alpha
\end{array}\right)+iy
\left(\begin{array}{cc}
v_1 & \frac{b}{2} \\
\frac{b^*}{2} & v_2
\end{array}\right)\right)\right\|
\left\|\left(\begin{array}{cc}
v_1 & \frac{b}{2} \\
\frac{b^*}{2} & v_2
\end{array}\right)^{-\frac12}\nonumber
\left(\begin{array}{cc}
0 & \frac{-ib}{2} \\
\frac{ib^*}{2} & 0
\end{array}\right)\right.\\
& & \mbox{}\times\left.
\left(\begin{array}{cc}
v_1 & \frac{b}{2} \\
\frac{b^*}{2} & v_2
\end{array}\right)^{-\frac12}
\right\|^2
\left\|\frac1y\Im f\left(\left(\begin{array}{cc}
\alpha & 0 \\
0 & \alpha
\end{array}\right)+iy
\left(\begin{array}{cc}
v_1 & {b} \\
0 & v_2
\end{array}\right)\right)
\right\|.
\end{eqnarray*}
We know from our hypothesis and Lemma \ref{lem:3.2} that the set of real positive numbers
$\left\{\left\|\frac1y\Im f\left(\left(\begin{array}{cc}
\alpha+iyv_1 & iyb \\
0 & \alpha+iyv_2
\end{array}\right)\right)\right\|\colon y\in(0,1)\right\}$ is bounded. If we assume that the set
$\left\{\left\|
\frac1y\Im f\left(\left(\begin{array}{cc}
\alpha+iyv_1 & \frac{iyb}{2} \\
\frac{iyb^*}{2} & \alpha+iyv_2
\end{array}\right)\right)\right\|\colon y\in(0,1)\right\}$ is unbounded and choose a sequence
$\{y_n\}_{n\in\mathbb N}$ converging to zero so that the strictly positive real number
$$
\ell:=\lim_{n\to\infty}\left\|\frac{1}{y_n}\Im f\left(\left(\begin{array}{cc}
\alpha+iy_nv_1 & iy_nb \\
0 & \alpha+iy_nv_2
\end{array}\right)\right)\right\|\text{ exists, and}
$$
$$
\lim_{n\to\infty}\left\|
\frac{1}{y_n}\Im f\left(\left(\begin{array}{cc}
\alpha+iy_nv_1 & \frac{iy_nb}{2} \\
\frac{iy_nb^*}{2} & \alpha+iy_nv_2
\end{array}\right)\right)\right\|=+\infty,
$$
then
\begin{eqnarray*}
\lefteqn{\left\|
\frac1{y_n}\Im f\left(\left(\begin{array}{cc}
\alpha+iy_nv_1 & \frac{iy_nb}{2} \\
\frac{iy_nb^*}{2} & \alpha+iy_nv_2
\end{array}\right)\right)\right\|-
\left\|\frac{1}{y_n}\Im f\left(\left(\begin{array}{cc}
\alpha+iy_nv_1 & iy_nb \\
0 & \alpha+iy_nv_2
\end{array}\right)\right)\right\|}\\
& \leq & \left\|\frac{1}{y_n}\Im f\left(\left(\begin{array}{cc}
\alpha & 0 \\
0 & \alpha
\end{array}\right)+iy_n
\left(\begin{array}{cc}
v_1 & \frac{b}{2} \\
\frac{b^*}{2} & v_2
\end{array}\right)\right)\right\|^\frac12
\left\|\left(\begin{array}{cc}
v_1 & \frac{b}{2} \\
\frac{b^*}{2} & v_2
\end{array}\right)^{-\frac12}\nonumber
\left(\begin{array}{cc}
0 & \frac{-ib}{2} \\
\frac{ib^*}{2} & 0
\end{array}\right)\right.\\
& & \mbox{}\times\left.
\left(\begin{array}{cc}
v_1 & \frac{b}{2} \\
\frac{b^*}{2} & v_2
\end{array}\right)^{-\frac12}
\right\|
\left\|\frac{1}{y_n}\Im f\left(\left(\begin{array}{cc}
\alpha & 0 \\
0 & \alpha
\end{array}\right)+iy_n
\left(\begin{array}{cc}
v_1 & {b} \\
0 & v_2
\end{array}\right)\right)
\right\|^\frac12
\end{eqnarray*}
becomes
\begin{eqnarray*}
\lefteqn{\left\|
\frac{\Im f\left(\left(\begin{array}{cc}
\alpha+iy_nv_1 & \frac{iy_nb}{2} \\
\frac{iy_nb^*}{2} & \alpha+iy_nv_2
\end{array}\right)\right)}{y_n}\right\|^\frac12-\frac{
\left\|\frac{1}{y_n}\Im f\left(\left(\begin{array}{cc}
\alpha+iy_nv_1 & iy_nb \\
0 & \alpha+iy_nv_2
\end{array}\right)\right)\right\|}{
\left\|
\frac1{y_n}\Im f\left(\left(\begin{array}{cc}
\alpha+iy_nv_1 & \frac{iy_nb}{2} \\
\frac{iy_nb^*}{2} & \alpha+iy_nv_2
\end{array}\right)\right)\right\|^\frac12}}\\
& \leq & 
\left\|\left(\begin{array}{cc}
v_1 & \frac{b}{2} \\
\frac{b^*}{2} & v_2
\end{array}\right)^{-\frac12}
\left(\begin{array}{cc}
0 & \frac{-ib}{2} \\
\frac{ib^*}{2} & 0
\end{array}\right)
\left(\begin{array}{cc}
v_1 & \frac{b}{2} \\
\frac{b^*}{2} & v_2
\end{array}\right)^{-\frac12}
\right\|\\
& & \mbox{}\times
\left\|\frac{1}{y_n}\Im f\left(\left(\begin{array}{cc}
\alpha & 0 \\
0 & \alpha
\end{array}\right)+iy_n
\left(\begin{array}{cc}
v_1 & {b} \\
0 & v_2
\end{array}\right)\right)
\right\|^\frac12;\quad\quad\quad\quad\quad\quad\quad\quad\quad\quad\quad\quad\quad
\end{eqnarray*}
by letting $n\to\infty$, we obtain 
$$
\infty-0\leq\ell\left\|\left(\begin{array}{cc}
v_1 & \frac{b}{2} \\
\frac{b^*}{2} & v_2
\end{array}\right)^{-\frac12}
\left(\begin{array}{cc}
0 & \frac{-ib}{2} \\
\frac{ib^*}{2} & 0
\end{array}\right)
\left(\begin{array}{cc}
v_1 & \frac{b}{2} \\
\frac{b^*}{2} & v_2
\end{array}\right)^{-\frac12}
\right\|,
$$ an obvious contradiction. We have thus shown that $\|\Im \mathfrak D\|/y$ stays
bounded as $y\searrow0$.
By relation \eqref{12}, the same holds for $\Re\mathfrak D.$ This proves the lemma.
\end{proof}

The previous lemma implies more: since $\left\|\frac1y\Im f\left(\left(\begin{array}{cc}
\alpha+iyv_1 & iyb \\
0 & \alpha+iyv_2
\end{array}\right)\right)\right\|$
is bounded as $y\in(0,1),$ it follows immediately 
from the lemma that
$$
\liminf_{y\downarrow0}\frac{1}{y}\varphi\left(\Im f\left(\left(\begin{array}{cc}
\alpha+iyv_1 & \frac{iyb}{2} \\
\frac{iyb^*}{2} & \alpha+iyv_2
\end{array}\right)\right)\right)<\infty, 
$$
for all states $\varphi$ on $M_2(\mathcal A)$, and so, as proved above,
\begin{equation}\label{2x2}
{\rm so-}\lim_{y\downarrow0}\frac{1}{y}\Im f\left(\left(\begin{array}{cc}
\alpha+iyv_1 & \frac{iyb}{2} \\
\frac{iyb^*}{2} & \alpha+iyv_2
\end{array}\right)\right):=C>0 \text{ in }M_2(\mathcal A).
\end{equation}
In particular, it follows that the finiteness of the liminf in 
\eqref{5} guarantees the boundedness of the sets 
$\Im f(\alpha\otimes 1_n+iyv)/y, y\in(0,1),$ for all $n\in\mathbb N$,
$v>0$ in $M_n(\mathcal A)$, and so the existence of the
corresponding so-limits for all $n$, as well as the norm-convergence of
$f(\alpha\otimes 1_n+zv)$ to $f(\alpha)\otimes 1_n$ as $z\to0$ nontangentially.

We show next the existence of the limit of $\Delta f(\alpha+iyv_1,\alpha+iyv_2)(b)$ as
$y\searrow0$. Let $v_1,v_2,b,\alpha$ be as in the above lemma. Fix $\epsilon\in(0,1)$
and denote $V_\epsilon=\left(\begin{array}{cc}
1 & 0 \\
0 & \sqrt\epsilon
\end{array}\right)$.
Observe that
$$
V_\epsilon^{-1}\left(\begin{array}{cc}
\alpha+iyv_1 & {iyb} \\
iy\epsilon b^* & \alpha+iyv_2
\end{array}\right)V_\epsilon=\left(\begin{array}{cc}
\alpha+iyv_1 & {i\sqrt\epsilon yb} \\
i\sqrt\epsilon yb^* & \alpha+iyv_2
\end{array}\right),
$$
so that, by the definition of a noncommutative function,
$$
f\left(\begin{array}{cc}
\alpha+iyv_1 & {iyb} \\
iy\epsilon b^* & \alpha+iyv_2
\end{array}\right)=V_\epsilon f\left(\begin{array}{cc}
\alpha+iyv_1 & {i\sqrt\epsilon yb} \\
i\sqrt\epsilon yb^* & \alpha+iyv_2
\end{array}\right)
V_\epsilon^{-1},
$$
The methods used in the proof of Lemma \ref{lem:4.2} allow for 
an estimate of the form
\begin{eqnarray*}
\lefteqn{\frac{1}{y^2}\left\|f\left(\begin{array}{cc}
\alpha+iyv_1 & {i\sqrt\epsilon yb} \\
i\sqrt\epsilon yb^* & \alpha+iyv_2
\end{array}\right)-\left(\begin{array}{cc}
f(\alpha) & 0 \\
0 & f(\alpha)
\end{array}\right)
\right\|^2}\\
& \leq & 
\left\|\left(\begin{array}{cc}
v_1 & 0 \\
0 & v_2
\end{array}\right)^{-\frac12}\left(\begin{array}{cc}
iv_1 & {i\sqrt\epsilon b} \\
i\sqrt\epsilon b^* & iv_2
\end{array}\right)
\left(\begin{array}{cc}
v_1 & \frac{\sqrt\epsilon b}{2} \\
\frac{\sqrt\epsilon b^*}{2} & v_2
\end{array}\right)^{-\frac12}\right\|^2\\
& & \mbox{}\times\left\|
\left(\begin{array}{cc}
c(v_1) & 0 \\
0 & c(v_2)
\end{array}\right)
\right\|
\left\|
\frac1y\Im f\left(\begin{array}{cc}
\alpha+iyv_1 & {i\sqrt\epsilon yb} \\
i\sqrt\epsilon yb^* & \alpha+iyv_2
\end{array}\right)
\right\|.
\end{eqnarray*}
If we denote $C_\epsilon:=\lim_{y\downarrow0}
\frac1y\Im f\left(\begin{array}{cc}
\alpha+iyv_1 & {i\sqrt\epsilon yb} \\
i\sqrt\epsilon yb^* & \alpha+iyv_2
\end{array}\right)$, the above allows us to conclude that
$$
\|C_\epsilon\|\leq\left\|\left(\begin{array}{cc}
v_1 & 0 \\
0 & v_2
\end{array}\right)^{-\frac12}\left(\begin{array}{cc}
v_1 & {\sqrt\epsilon b} \\
\sqrt\epsilon b^* & v_2
\end{array}\right)
\left(\begin{array}{cc}
v_1 & \frac{\sqrt\epsilon b}{2} \\
\frac{\sqrt\epsilon b^*}{2} & v_2
\end{array}\right)^{-\frac12}\right\|^2\max_{1\le j\le 2}\|c(v_j)\|.
$$
Thus, for any $\epsilon\in(0,1)$ we have $\|C_\epsilon\|\leq\textrm{const}(v_1,v_2,b)$.
However, a bit more can be obtained: since conjugation by $V_\epsilon$ 
does not affect diagonal elements, we have 
\begin{eqnarray*}
\lefteqn{\frac{1}{y^2}\left\|f\left(\begin{array}{cc}
\alpha+iyv_1 & {i yb} \\
i\epsilon yb^* & \alpha+iyv_2
\end{array}\right)-\left(\begin{array}{cc}
f(\alpha) & 0 \\
0 & f(\alpha)
\end{array}\right)
\right\|^2}\\
& = & \frac{1}{y^2}\left\|V_\epsilon \left(f\left(\begin{array}{cc}
\alpha+iyv_1 & {i\sqrt\epsilon yb} \\
i\sqrt\epsilon yb^* & \alpha+iyv_2
\end{array}\right)-\left(\begin{array}{cc}
f(\alpha) & 0 \\
0 & f(\alpha)
\end{array}\right)\right)V_\epsilon^{-1}
\right\|^2\\
& \leq & \frac1\epsilon
\left\|\left(\begin{array}{cc}
v_1 & 0 \\
0 & v_2
\end{array}\right)^{-\frac12}\left(\begin{array}{cc}
iv_1 & {i\sqrt\epsilon b} \\
i\sqrt\epsilon b^* & iv_2
\end{array}\right)
\left(\begin{array}{cc}
v_1 & \frac{\sqrt\epsilon b}{2} \\
\frac{\sqrt\epsilon b^*}{2} & v_2
\end{array}\right)^{-\frac12}\right\|^2\\
& & \mbox{}\times\left\|
\left(\begin{array}{cc}
c(v_1) & 0 \\
0 & c(v_2)
\end{array}\right)
\right\|
\left\|
\frac1y\Im f\left(\begin{array}{cc}
\alpha+iyv_1 & {i\sqrt\epsilon yb} \\
i\sqrt\epsilon yb^* & \alpha+iyv_2
\end{array}\right)
\right\|,
\end{eqnarray*}
as $\|V_\epsilon\|=1,\|V_\epsilon^{-1}\|=\epsilon^{-1/2}$. 
The existence of the limit 
$$
\ell_\epsilon:=\lim_{y\downarrow0}\frac1y\left[
f\left(\begin{array}{cc}
\alpha+iyv_1 & {i\sqrt\epsilon yb} \\
i\sqrt\epsilon yb^* & \alpha+iyv_2
\end{array}\right)-\left(\begin{array}{cc}
f(\alpha) & 0 \\
0 & f(\alpha)
\end{array}\right)\right]
$$
implies the existence of
$$
\lim_{y\downarrow0}\frac1y\left[f\left(\begin{array}{cc}
\alpha+iyv_1 & {i yb} \\
i\epsilon yb^* & \alpha+iyv_2
\end{array}\right)-\left(\begin{array}{cc}
f(\alpha) & 0 \\
0 & f(\alpha)
\end{array}\right)\right]=V_\epsilon\ell_\epsilon V_\epsilon^{-1}.
$$
Let us now continue our estimates on the derivative:
\begin{eqnarray*}
\lefteqn{\frac{1}{y^2}\left\|f\left(\begin{array}{cc}
\alpha+iyv_1 & {i yb} \\
i\epsilon yb^* & \alpha+iyv_2
\end{array}\right)-f\left(\begin{array}{cc}
\alpha+iyv_1 & iyb \\
0 & \alpha+iyv_2
\end{array}\right)
\right\|^2}\\
& \leq & 
\left\|\left(\begin{array}{cc}
yv_1 & \frac{yb}{2} \\
\frac{yb^*}{2}  & yv_2
\end{array}\right)^{-\frac12}\left(\begin{array}{cc}
0 & 0 \\
i\epsilon yb^* & 0
\end{array}\right)
\left(\begin{array}{cc}
yv_1 & \frac{(1+\epsilon)yb}{2} \\
\frac{(1+\epsilon)y b^*}{2} & yv_2
\end{array}\right)^{-\frac12}\right\|^2\\
& & \mbox{}\times\left\|
\frac1y\Im f\left(\begin{array}{cc}
\alpha+iyv_1 & iyb \\
0 & \alpha+iyv_2
\end{array}\right)
\right\|
\left\|
\frac1y\Im f\left(\begin{array}{cc}
\alpha+iyv_1 & {iyb} \\
i\epsilon yb^* & \alpha+iyv_2
\end{array}\right)
\right\|.
\end{eqnarray*}
The first factor on the right hand side is bounded 
by $\epsilon^2\textrm{const}(b,v_1,v_2)$, for a constant 
$\textrm{const}(b,v_1,v_2)\in\mathbb R$, independent of 
$y,\epsilon\in(0,1)$. The second factor has been
shown in Lemma \ref{lem:3.2} to be 
bounded uniformly in $y\in(0,1)$.
Finally, the last term is dominated, as seen above, by
$\epsilon^{-1}\textrm{const}(b,v_1,v_2)$. Thus,
$$
\frac{1}{y^2}\left\|f\left(\begin{array}{cc}
\alpha+iyv_1 & {i yb} \\
i\epsilon yb^* & \alpha+iyv_2
\end{array}\right)-f\left(\begin{array}{cc}
\alpha+iyv_1 & iyb \\
0 & \alpha+iyv_2
\end{array}\right)
\right\|^2\leq\epsilon\textrm{const}(v_1,v_2,b),
$$
for any $y,\epsilon\in(0,1)$. By weak compactness of norm-bounded sets,
any sequence tending to zero has a subsequence $\{y_n\}$ such that
${\displaystyle\lim_{n\to\infty}}\Delta f(\alpha+iy_nv_1,\alpha+iy_nv_2)(b)$
exists in the weak operator topology.
Adding and subtracting $\left(\begin{array}{cc}
f(\alpha) & 0 \\
0 & f(\alpha) 
\end{array}\right)$ under the norm in the left hand side above
and letting $y\searrow0$ along such a sequence provides 
$$
\left\|V_\epsilon\ell_\epsilon V_\epsilon^{-1}-\left(\begin{array}{cc}
c(v_1) & {\displaystyle\lim_{n\to\infty}}\Delta f(\alpha+iy_nv_1,\alpha+iy_nv_2)(b) \\
0 & c(v_2) 
\end{array}\right)
\right\|\leq\sqrt{\epsilon}\textrm{const}(v_1,v_2,b),
$$
for any fixed $\epsilon\in(0,1)$. This restricts the diameter
of the cluster set of $\Delta f(\alpha+iyv_1,\alpha+iyv_2)(b)$ at zero
to a set of norm-diameter of order $\sqrt\epsilon$ for any $\epsilon>0$.
Thus, the limit ${\displaystyle\lim_{y\downarrow0}}\Delta f(\alpha+iyv_1,\alpha+iyv_2)(b)$
must exist.

We conclude that $\displaystyle\lim_{y\downarrow0}\Delta f(\alpha+iyv_1,\alpha+iyv_2)(b)$
exists and is uniformly bounded as $b\in\mathcal A$ stays in a bounded subset of $\mathcal A$.
Clearly the limit depends linearly on $b$, since each of $\Delta f(\alpha+iyv_1,\alpha+iyv_2)(b)$
does. In particular, if $v_1=v_2=v$, $\Delta f(\alpha+iyv,\alpha+iyv)(b)=f'(\alpha+iyv)(b)$ has a 
limit as $y\to0$, as claimed in part (1) of Theorem \ref{Main}. Let now in addition $b=v/4$.
For any state $\varphi$ on $\mathcal A$ and $v>0$, $z\mapsto\varphi(f(\alpha+zv))$ is a self-map
of $\mathbb C^+$ which satisfies the conditions of Theorem \ref{JC} at $z=0$. Thus,
$\displaystyle\lim_{y\downarrow0}\varphi(f'(\alpha+iyv)(v))=
\lim_{y\downarrow0}\frac{\varphi(\Im f(\alpha+iyv))}{y}$, so that indeed
$$
\lim_{y\downarrow0}f'(\alpha+iyv)(v)=\lim_{y\downarrow0}\frac{\Im f(\alpha+iyv)}{y}=c(v)>0.
$$

Until now we have proved that the finiteness of the liminf in \eqref{5} (which is applied to elements
in $\mathcal A=M_1(\mathcal A)$) implies not only the existence of $f(\alpha)$ and of limits
of $\Delta f(\alpha+iyv_1,\alpha+iyv_2)$ as $y\downarrow0$, but also the existence and finiteness
of the liminf in \eqref{5} applied to $\alpha$ replaced by $\alpha\otimes 1_{M_2(\mathbb C)}$
and $v$ replaced by a positive in $M_2(\mathcal A)$. Obviously, we now apply the above
results to elements in $M_2(\mathcal A)$ to obtain the same conclusion for elements in 
$M_4(\mathcal A)$ and so on. This, according to \cite[Chapters 2 and 3]{ncfound}, allows us to 
conclude the proof of part (1).

We prove next part (1') of Theorem \ref{Main}. Let $v,w>0$ 
be fixed. Recall that we have shown in the proof of part (1) that 
$\displaystyle\lim_{t\downarrow0}f'(\alpha+ity_1v+ity_2w)$ exists 
pointwise. Our hypothesis that 
$$
\lim_{y_1,y_2\to0}
(\varphi(f'(\alpha+iy_1v+iy_2w)(v)),\varphi(f'(\alpha+iy_1v+iy_2w)(w)))
$$
exists and is finite for any state $\varphi$ on $\mathcal A$
implies that $f'(\alpha+iy_1v+iy_2w)(v),f'(\alpha+iy_1v+iy_2w)(w)$ have a weak limit 
as $(y_1,y_2)\downarrow(0,0)$ in $[0,1)^2\setminus\{(0,0)\}$. Note that the domain $\{(z,\zeta)\in
\mathbb C^2\colon\Im(zv+\zeta w)>0\}$ of the function 
$(z,\zeta)\mapsto\varphi(f(\alpha+zv+\zeta w))$ includes $\overline{\mathbb C^+}\times\mathbb C^+
\cup\mathbb C^+\times\overline{\mathbb C^+}$ (closures taken in $\mathbb C$).
In particular, $\{(z,0)\colon z\in\mathbb C^+\}\cup\{(0,\zeta)\colon\zeta\in\mathbb C^+\}\subset
\{(z,\zeta)\in\mathbb C^2\colon\Im(zv+\zeta w)>0\}.$
The existence of the above displayed limit thus guarantees that
$\lim_{y\downarrow0}\varphi(f'(\alpha+iyw)(v))=
\lim_{y\downarrow0}\varphi(f'(\alpha+iyv)(v))$. This means that
the limit of $f'(\alpha+iyv)$ as $y\downarrow0$ does not depend on $v$ and
is positive. Applying this same result to $M_n(\mathcal A)$ and recalling the properties
of noncommutative functions guarantee complete positivity for $f'(\alpha)$. To
conclude the proof of (1'), simply observe that $\Delta f(\alpha+iyv_1,\alpha+iyv_2)(b)-
f'(\alpha+iyv_1)(b)$ converges to zero as $y\downarrow0$.

The proof of (2) is much simpler. Indeed, the existence of the limit 
$\lim_{y\downarrow0}f'(\alpha+iyv)$ implies the existence of the limit
$\lim_{y\downarrow0}\varphi(f'(\alpha+iyv)(v))$ for all states $\varphi$ on $\mathcal A$.
An application of Theorem \ref{JC} and of parts (1) and (1') of Theorem \ref{Main}
allows us to conclude.
\end{proof}

It might be useful to note that the operator $C$ from equality \eqref{2x2} can be written in terms of the
small $c$'s form the statement of Theorem \ref{Main}, at least when
$v_1=v_2$. We use here the
condition (A) of the definition of noncommutative functions. Let $v>0$
be fixed and let $b$ be such that $v>b>0$ in $\mathcal A$.
Then
$$
\left(\begin{array}{cc}
\alpha+iyv & iyb \\
iyb & \alpha+iyv
\end{array}\right)\left(\begin{array}{ccc}
1 & 0 & 1 \\
1 & 1 & 0 
\end{array}\right)=
\left(\begin{array}{ccc}
\alpha+iy(v+b) & iyb &\alpha+iyv \\
\alpha+iy(v+b) & \alpha+iyv & iyb
\end{array}\right),
$$
which is in its own turn equal to the product
$$
\left(\begin{array}{ccc}
1 & 0 & 1 \\
1 & 1 & 0 
\end{array}\right)\left(\begin{array}{ccc}
\alpha+iy(v+b) & iyb & 0 \\
0 & \alpha+iy(v-b)  & iyb \\
0 & 0 & \alpha+iyv 
\end{array}\right).
$$
We recognize in the $2\times2$ matrix the argument of one of the 
terms involved in the statement of Lemma \ref{lem:4.2}. If we denote
$$
f\left(\begin{array}{cc}
\alpha+iyv & iyb \\
iyb & \alpha+iyv
\end{array}\right)=\left(\begin{array}{cc}
f_{11} & f_{12} \\
f_{21} & f_{22}
\end{array}\right),
$$
then condition (A) tells us that
{\small \begin{eqnarray*}
\lefteqn{\left(\begin{array}{ccc}
f_{11}+f_{12} & f_{12} & f_{11} \\
f_{21}+f_{22} & f_{22} & f_{21}
\end{array}\right)=\left(\begin{array}{ccc}
1 & 0 & 1 \\
1 & 1 & 0 
\end{array}\right)\times}\\
& & 
\left(\begin{array}{ccc}
f(\alpha+iy(v+b)) & \frac{f(\alpha+iy(v+b))-f(\alpha+iy(v-b))}{2} & \Delta^2f \\
0 & f(\alpha+iy(v-b)) & f(\alpha+iyv)-f(\alpha+iy(v-b)) \\
0 & 0 & f(\alpha+iyv)
\end{array}\right)= \\
& & \left(\begin{array}{ccc}
f(\alpha+iy(v+b)) & \frac{f(\alpha+iy(v+b))-f(\alpha+iy(v-b))}{2} & \Delta^2f+f(\alpha+iyv) \\
f(\alpha+iy(v+b)) & \frac{f(\alpha+iy(v+b))+f(\alpha+iy(v-b))}{2} & f(\alpha+iyv)-f(\alpha+iy(v-b))+
\Delta^2f
\end{array}\right),
\end{eqnarray*}}
where $\Delta^2f$ stands for $\Delta^2f(\alpha+iy(v+b),\alpha+iy(v-b),\alpha+iyv)
( iyb , iyb )$.
We obtain immediately the relations
\begin{eqnarray*}
f_{11}=f_{22}&=&\frac{f(\alpha+iy(v+b))+f(\alpha+iy(v-b))}{2}\\
f_{21}=f_{12}&=&\frac{f(\alpha+iy(v+b))-f(\alpha+iy(v-b))}{2}.
\end{eqnarray*}
It follows that, for $v_1=v_2>0$,
\begin{eqnarray*}
C & = & \lim_{y\downarrow0}\frac1y\Im f\left(\left(\begin{array}{cc}
\alpha+iy & iyb \\
iyb & \alpha+iyv
\end{array}\right)\right)\\
& = & \frac12\lim_{y\downarrow0}
\left(\begin{array}{cc}
\frac{\Im f(\alpha+iy(v+b))+\Im f(\alpha+iy(v-b))}{y}&\frac{\Im f(\alpha+iy(v+b))-\Im f(\alpha+iy(v-b))}{y} 
\\
\frac{\Im f(\alpha+iy(v+b))-\Im f(\alpha+iy(v-b))}{y}&\frac{\Im f(\alpha+iy(v+b))+\Im f(\alpha+iy(v-b))}{y}
\end{array}\right).
\end{eqnarray*}
By considering the functions $z\mapsto\varphi( f(\alpha+z(v\pm b)))$, we obtain
on the off-diagonal entries precisely $[f'(\alpha)(v+b)-f'(\alpha)(v-b)]/2$ and 
on the diagonal entries $[f'(\alpha)(v+b)+f'(\alpha)(v-b)]/2$.

Moreover, the set of elements $b\in\mathcal A$ such that
$0<b<v$ is open in the set of selfadjoints, and the set of selfadjoints is a set of
uniqueness for analytic maps. Thus, the above formulas for $f_{ij}$ hold for any
$b$ from the connected component of the domain of the maps in question (viewed
as functions of $b$).

During the inception and elaboration of this paper I had the privilege to discuss various
aspects related to it with Hari Bercovici, Victor Vinnikov and Gilles Pisier.
I thank them very much both for valuable advices and encouragements. 
I would also like to thank Marco Abate for discussions on the first draft of this
paper that motivated me to expand it. Clearly, any mistakes
are entirely mine.

\end{document}